\documentclass{amsart}

\usepackage{lineno}
\usepackage{graphicx}
\usepackage{mathrsfs}
\usepackage{stmaryrd}
\usepackage{amsfonts}
\usepackage{enumerate,amsmath,amssymb,amsthm}

\renewcommand{\theequation}{\arabic{equation}}
\newcommand{\be}{\begin{eqnarray}}
\newcommand{\ee}{\end{eqnarray}}
\newcommand{\ce}{\begin{eqnarray*}}
\newcommand{\de}{\end{eqnarray*}}
\newtheorem{theorem}{Theorem}[section]
\newtheorem{lemma}[theorem]{Lemma}
\newtheorem{remark}[theorem]{Remark}
\newtheorem{definition}[theorem]{Definition}
\newtheorem{proposition}[theorem]{Proposition}
\newtheorem{Examples}[theorem]{Examples}
\newtheorem{corollary}[theorem]{Corollary}

\def\[{{\Big[}}
\def\]{{\Big]}}
\def\<{{\langle}}
\def\>{{\rangle}}
\def\({{\Big(}}
\def\){{\Big)}}

\def\bx{{\mathbf{x}}}

\def\bt{\begin{theorem}}
\def\et{\end{theorem}}
\def\bl{\begin{lemma}}
\def\el{\end{lemma}}
\def\br{\begin{remark}}
\def\er{\end{remark}}
\def\bx{\begin{Examples}}
\def\ex{\end{Examples}}
\def\bd{\begin{definition}}
\def\ed{\end{definition}}
\def\bp{\begin{proposition}}
\def\ep{\end{proposition}}
\def\bc{\begin{corollary}}
\def\ec{\end{corollary}}

\def\geq{\geqslant}
\def\leq{\leqslant}

\theoremstyle{thmit} % Numbered and Italic

\theoremstyle{thmrm} % Numbered and Roman

\newtheorem*{oldproof}{Proof}
\renewenvironment{proof}[1][{}]{\begin{oldproof}[#1]}{\qed\end{oldproof}}

\title{S\MakeLowercase{low manifolds for a nonlocal fast-slow stochastic evolutionary system with stable} L$\acute{\MakeLowercase{e}}$\MakeLowercase{vy noise}\footnote{The research   was partly supported  by the NSF   grant 1620449   and NSFC grants  11531006 and  11771449.}}

\author{H\MakeLowercase{ina} Z\MakeLowercase{ulfiqar}}
\address{School of Mathematics and Statistics,
Huazhong University of Science and Technology,\\ Wuhan 430074, China \\ \& Center for  Mathematical Sciences, Huazhong University of Science and Technology\\zhinazulfiqar@gmail.com}
\author{S\MakeLowercase{henglan} Y\MakeLowercase{uan}}
\address{School of Mathematics and Statistics,
Huazhong University of Science and Technology,\\ Wuhan 430074, China \\ \& Center for  Mathematical Sciences, Huazhong University of Science and Technology\\shenglanyuan@hust.edu.cn}
\author{Z\MakeLowercase{iying} H\MakeLowercase{e}}
\address{School of Mathematics and Statistics,
Huazhong University of Science and Technology,\\ Wuhan 430074, China \\ \& Center for  Mathematical Sciences, Huazhong University of Science and Technology\\ziyinghe@hust.edu.cn}
\author{J\MakeLowercase{inqiao} D\MakeLowercase{uan}}
\address{Department of Applied Mathematics, Illinois Institute of Technology, Chicago, IL 60616,\\ USA   \\duan@iit.edu }
\begin{document}
%\linenumbers

\maketitle
\begin{abstract}
% Abstract text
 \indent This work aims at understanding the slow dynamics of a nonlocal fast-slow stochastic evolutionary system with stable L$\acute{e}$vy noise.   Slow manifolds along with exponential tracking property for a nonlocal fast-slow stochastic evolutionary system with stable L$\acute{e}$vy noise are constructed  and two examples with numerical simulations are presented to illustrate the results. 
\end{abstract}

\section{Introduction}\label{s:1}
\begin{linenomath*}
Over the last few years, the theory of nonlocal operators attracts a lot of attention from researchers because most of the complex phenomena \cite{caffarelli2010drift,meerschaert2012stochastic,metzler2004restaurant} involve nonlocal operators. Many researchers made a lot of progress by working on different type of nonlocal operators. The usual Laplacian operator $\Delta$ is not a nonlocal operator. It generates Brownian motion (or Wiener process), which is Gaussian process. While nonlocal Laplacian operator $(-\Delta)^{\frac{\alpha}{2}}$ generates a symmetric $\alpha$-stable L$\acute{e}$vy motion, for $\alpha\in(0,2)$, \cite{applebaum2009levy, duan2015introduction}. This motion is non-Gaussian process.  \\
\indent The theory of invariant manifolds is very helpful for describing and understanding dynamics of deterministic systems under stochastic forces. It was introduced in \cite{hadamard1901iteration,caraballo2004existence,duan2004smooth,chow1988invariant}, while for deterministic system its modification was given in \cite{ruelle1982characteristic,bates1998existence,chicone1997center,chow1991smooth,henry2006lecture} by  numerous authors.\\
There is very rich and papular history for the theory of invariant manifold \cite{bates1998existence,henry2006lecture} in finite and infinite deterministic systems. Furthermore, invariant manifold provides us very helpful tool in investigating the dynamical conduct of stochastic systems \cite{chueshov2010master,chen2014slow,duan2004smooth}. An invariant manifold for a fast-slow stochastic system in which fast mode is indicated by the slow mode tends to slow manifold as scale parameter approaches to zero. Moreover, slow manifold for a fast-slow stochastic system tends to critical manifold as scale parameter approaches to zero.\\
\indent The existence of slow manifold for stochastic system based on Brownian motion has been widely constructed \cite{duan2015introduction,fu2012slow,schmalfuss2008invariant,wang2013slow}. The numerical simulation for slow manifold and establishment of parameter estimation are provided in \cite{ren2015approximation,ren2015parameter}. L$\acute{e}$vy motions appear in many systems as models for fluctuations, for instance, it appear in the turbulent motions of fluid flows \cite{weeks1995observation}. A few monographs about stochastic ordinary differential equations processed by L$\acute{e}$vy noise are devoted in \cite{applebaum2009levy,cont2004option}. The existence of slow manifold under non-Gaussian L$\acute{e}$vy noise is constructed in \cite{yuan2017slow}. While the study of dynamics for nonlocal stochastic differential equations processed by non-Gaussian L$\acute{e}$vy noise is still under development. \\
\indent The main objective of this article is to construct the existence of slow manifold for a nonlocal stochastic dynamical system processed by $\alpha$-stable L$\acute{e}$vy noise with $\alpha \in (1,2)$ defined in a separable Hilbert space $\mathbb{H}=H_{1}\times H_{2}$ having norm
\begin{align*}||\cdot||_{\mathbb{H}}=||\cdot||_{1}+||\cdot||_{2}.\end{align*} Namely, we consider the system
\begin{align} &\dot{x}=-\frac{1}{\epsilon}(-\Delta)^{\frac{\alpha}{2}}x+\frac{1}{\epsilon}f(x,y)+\frac{\sigma_{1}}{\sqrt[\alpha_{1}]{\epsilon}}\dot{L}_{t}^{\alpha_{1}},\mbox{ in }H_{1}\\&\dot{y}=Jy+g(x,y)+\sigma_{2}\dot{L}_{t}^{\alpha_{2}},\mbox{ in }H_{2}\\ &x|(-1,1)^{c}=0,\indent \indent y|(-1,1)^{c}=0. \end{align}Here, for $u\in \mathbb{R}$ and $\alpha\in(0,2),$\\$$(-\Delta)^{\frac{\alpha}{2}}x(u,t)=\frac{2^{\alpha}\Gamma(\frac{1+\alpha}{2})}{\sqrt{\pi}|\Gamma(\frac{-\alpha}{2})|}P.V.\int_{\mathbb{R}}\frac{x(u,t)-x(v,t)}{|u-v|^{1+\alpha}}dv,$$ \\is known as fractional Laplacian operator with the Cauchy principle value $(P.V.)$. The Gamma function $\Gamma$ is defined by\\$$\Gamma(q)=\int_{0}^{\infty}t^{q-1}e^{-t}dt,\indent \forall \indent q>0.$$\indent We take $H_{1}=L^{2}(-1,1)$ and $H_{2}$ a separable Hilbert space. The norm of $H_{1}$ and $H_{2}$ are $||\cdot||_{1}$ and $||\cdot||_{2}$ respectively. In the system $(1)-(3)$, $\epsilon$ is a parameter with the property $0<\epsilon\ll1$. This parameter represents the ratio of two times scales such that $||\frac{dx}{dt}||_{1}\gg||\frac{dy}{dt}||_{2}.$ The operator $J$ is linear operator satisfying an exponential dichotomy condition (S1) presented in next section. Lipschitz continuous  operators $f$ and $g$ are nonlinear with $f(0,0)=0=g(0,0)$. The noise process $L_{t}^{\alpha}$ are two sided symmetric $\alpha$-stable L$\acute{e}$vy process taking values in Hilbert space $\mathbb{H}$, where $\alpha \in (1,2)$ is the index of stability \cite{applebaum2009levy,chow1991smooth}. \end{linenomath*}
\\
\indent We introduce a random transformation such that a solution of stochastic dynamical system $(1)-(3)$ can be indicated as a transformed solution of some random dynamical system. After that, we establish the construction of slow manifold for random dynamical system with the help of Lyapunov-Perron method \cite{caraballo2004existence,duan2004smooth,chow1988invariant}.\\
\indent The setup of this article is as follows. In Section 2, some fundamental concepts about random dynamical system, nonlocal fractional Laplacian and a detail discussion about differential equation processed by L$\acute{e}$vy motion are given. In Section 3, we convert stochastic dynamical system $(1)-(3)$ to random dynamical system by introducing a random transformation. In Section 4, we review concept about random invariant manifold and establish the existence of exponential tracking slow manifold for random dynamical system. In section 5, an approximation to slow manifold is established. While in Section 6, two examples with numerical simulations are presented to illustrate the results.

\section{Preliminaries}
In this section we recall out some ideas about fractional Laplacian operator and random dynamical system processed by L$\acute{e}$vy motion.\\
\indent The nonlocal fractional Laplacian operator is represented by $A_{\alpha}$ and considered as $A_{\alpha}=-(-\Delta)^{\frac{\alpha}{2}}$.\\
\begin{lemma}\begin{linenomath*} (\cite{bai2017slow}) The fractional Laplacian operator $A_{\alpha} $ has the upper-bound$$||e^{A_{\alpha}t}||_{1}\leq C e^{-\lambda_{1}t},\indent t\geq0,$$where the constant $C>0$  is independent of $t$ and $\lambda_{1}$. Nonlocal fractional Laplacian operator is also known as a sectorial operator.\\\end{linenomath*}\end{lemma}
\begin{lemma} \begin{linenomath*}(\cite{kwasnicki2012eigenvalues}) The spectral problem$$(-\Delta)^{\frac{\alpha}{2}}\varphi(u)=\lambda\varphi(u),\indent \varphi|(-1,1)^{c}=0,$$where $\varphi(\cdot)\in H_{1}$ are defined in (\cite{kwasnicki2012eigenvalues}), has eigenvalues in the interval (-1,1) satisfying the form $$\lambda_{l}=(\frac{l\pi}{2}-\frac{(2-\alpha)\pi}{8})^{\alpha}+O(\frac{1}{l}),\indent (l\rightarrow\infty).$$Furthermore the eigenvalues of fractional Laplacian are such that,$$0<\lambda_{1}<\lambda_{2}\leq \lambda_{3}\leq\cdot\cdot\cdot\leq\lambda_{l}\leq \cdot\cdot\cdot, \mbox{ for } l=1,2,3,\cdot\cdot\cdot.$$\end{linenomath*}\end{lemma}
\bd (\cite{yuan2017slow})
\begin{linenomath*}Let $(\Omega, \mathcal{F},\mathbb{P})$ be a probability space and $\theta=\{\theta_{l}\}_{l\in\mathbb{R}}$ be a flow on $\Omega$ such that\\
 $\bullet$  $\theta_{0}=Id_{\Omega};$\\
   $\bullet$ $\theta_{l_{1}}\theta_{l_{2}}=\theta_{l_{1}+l_{2}},$ where $l_{1},l_{2}\in\mathbb{R};$\\and it can be defined by a mapping $$\theta:\mathbb{R}\times\Omega\rightarrow\Omega.$$
   The above mapping $(l,\omega)\mapsto \theta_{l}\omega$ is $(\mathcal{B(\mathbb{R})\otimes\mathcal{F},F})$-measurable, and $\theta_{l}\mathbb{P}=\mathbb{P}$ for all $l\in\mathbb{R}$. Here additionally we consider that the probability measure $\mathbb{P}$ is invariant with regard to the flow $\{\theta_{l}\}_{l\in\mathbb{R}}$. Then $\Theta=\(\Omega, \mathcal{F},\mathbb{P},\theta)$ is known as a metric dynamical system.\end{linenomath*}\ed
\begin{linenomath*} In this work, let $L_{t}^{\alpha}$, $\alpha\in (1,2)$ be a two sided symmetric $\alpha$-stable L$\acute{e}$vy process having values in Hilbert space
 $\mathbb{H}$. Take a canonical sample space for two sided symmetric $\alpha$-stable L$\acute{e}$vy process.
 Let $\Omega=D(\tilde{K},\mathbb{H})$ be the space of c$\grave{a}$dl$\grave{a}$g functions, having zero value at $t=0$. These functions are defined on compact subset $\tilde{K}$ of $\mathbb{R}$ and taken values in Hilbert space $\mathbb{H}$. If we use the usual open-compact metric, then the space $D(\tilde{K},\mathbb{H})$ may not separable and complete. The space can be made complete and separable by defining another metric $d_{\tilde{K}}^{0}$ just as the space of real valued c$\grave{a}$dl$\grave{a}$g functions can be made complete and separable on unit interval or on $\mathbb{R}$  \cite{wei2016weak,chao2018stable}. For making space $D(\tilde{K},\mathbb{H})$ complete and separable, let $D^{0}(\tilde{K},\mathbb{H})$ be the subset of $D(\tilde{K},\mathbb{H})$ as defined in definition 3.6 of \cite{wei2016weak}. Hence, the class of functions denoted by $\Lambda_{\tilde{K}}^{0}$ with respect to new metric is
 $$\Lambda_{\tilde{K}}^{0}=\Big\{\mbox{ mapping }\lambda:\tilde{K}\rightarrow\tilde{K} \mbox{ is a strictly increasing and continuous function}\Big\}.$$
  Then $d_{\tilde{K}}^{0}$ corresponding to class $\Lambda_{\tilde{K}}^{0}$ is given by
 \begin{align*}d_{\tilde{K}}^{0}(f_{1},f_{2})=&\mathop {\mbox{inf} }\limits_{\lambda \in \Lambda_{\tilde{K}}^{0} }\mbox{ max }\bigg \{\mathop {\mbox{ sup } }\limits_{x>x^{*}, x, x^{*} \in \tilde{K}}\mbox{ log }\left|\left[ \frac{||\lambda(x)||_{\mathbb{R}}-\lambda(x^{*})||_{\mathbb{R}}}{||x||_{\mathbb{R}}-||x^{*}||_{\mathbb{R}}}\right]\right|,\\&||\lambda-I||_{\mbox{sup}},||f_{1}-f_{2}\lambda||_{\mbox{sup}}\bigg \},\end{align*}
 for $f_{1}, f_{2}$ in $D^{0}(\tilde{K},\mathbb{H})$. \\By Theorem 3.2 in \cite{wei2016weak}, the metric space $[D^{0}(\tilde{K},\mathbb{H}),d_{\tilde{K}}^{0}]$ is complete and separable. Hence, the class of functions $D^{0}(\tilde{K},\mathbb{H})$ is equipped with Skorokhod's topology, which is generated by Skorokhod's metric $d_{\tilde{K}}^{0}$, is a Polish space, i.e., a complete and separable space. On this space, take a measurable flow $\theta=\{\theta_{l}\}_{l\in\tilde{K}}$ is defined namely a mapping
 \begin{align*}
 \theta:\tilde{K}\times D^{0}(\tilde{K},\mathbb{H})\rightarrow D^{0}(\tilde{K},\mathbb{H}),\mbox{ such that}, \theta_{l}\omega(\cdot)=\omega(\cdot+l)-\omega(l),\end{align*}
 where $\omega\in D^{0}(\tilde{K},\mathbb{H})$ and $l \in \tilde{K}$.\\
\indent Suppose that $\mathbb{P}$ be the probability measure on $\mathcal{F}$ defined by the distribution of two sided symmetric $\alpha$-stable L$\acute{e}$vy motion. The sample path of L$\acute{e}$vy motion are in $D(\tilde{K},\mathbb{H})$. Note that $\mathbb{P}$ is ergodic with regard to $\{\theta_{l}\}_{l\in\tilde{K}}$. Thus $( D^{0}(\tilde{K},\mathbb{H}), d_{\tilde{K}}^{0}, \mathbb{P}, \{\theta_{l}\}_{l\in\tilde{K}})$ is a metric dynamical system. Instead of considering $D(\tilde{K},\mathbb{H})$, here we consider $D^{0}(\tilde{K},\mathbb{H})$, a $\{\theta_{l}\}_{l\in\tilde{K}}$-invariant subset $\Omega_{1}=D^{0}(\tilde{K},\mathbb{H})\subset \Omega=D(\tilde{K},\mathbb{H})$ of $\mathbb{P}$-measure 1, where $D^{0}(\tilde{K},\mathbb{H})$ is $\{\theta_{l}\}_{l\in\tilde{K}}$-invariant mean that $\theta_{l}\Omega_{1}=\Omega_{1}$ for $l\in \tilde{K}$. Since on $\mathcal{F}$, we take the restriction of measure $\mathbb{P}$, but still it is denoted by $\mathbb{P}$. For our project, we take scalar L$\acute{e}$vy motion under consideration.\end{linenomath*}
\bd (\cite{arnold2013random})
 \begin{linenomath*}A cocycle $\phi$ satisfies \begin{align*}&\phi(0,\omega,x)=x,\\&\phi(l_{1}+l_{2},\omega,x)=\phi(l_{2},\theta_{l_{1}}\omega,\phi(l_{1},\omega,x)).\end{align*}It is $(\mathcal{B(\mathbb{R^{+}})\otimes\mathcal{F}\otimes\mathcal{B(\mathbb{H})},\mathcal{F}})$-measurable and defined by mapping:$$\phi:\mathbb{R^{+}}\times\Omega\times\mathbb{H}\rightarrow\mathbb{H},$$ for $x\in\mathbb{H}$, $\omega\in\Omega$ and $l_{1},l_{2}\in\mathbb{R^{+}}$. Metric dynamical system $(\Omega,\mathcal{F},\mathbb{P},\theta)$, together with $\phi$, generates a random dynamical system.\end{linenomath*}\ed
\begin{linenomath*} If $x\mapsto\phi(l,\omega,x)$ is continuous (differentiable) for $\omega\in\Omega$ and $l\geq0$, then random dynamical system is continuous (differentiable).
There is a family of non-empty and closed sets $\mathcal{M}=\{\mathcal{M}(\omega):\omega\in \Omega\}$ in metric space $(\mathbb{H},||\cdot||_{\mathbb{H}})$. This family of sets is called a random set if for all $x'\in\mathbb{H}$ the map:$$\omega\mapsto\mathop {\inf }\limits_{x \in \mathcal{M}(\omega)}||x-x'||_{\mathbb{H}},$$ is a random variable.\end{linenomath*}
\bd
\begin{linenomath*}(\cite{duan2015introduction}) For a random dynamical system $\phi$, if random variable $x(\omega)$ taking values in $\mathbb{H}$ satisfies $$\phi(l,\omega,x(\omega))=x(\theta_{l}\omega),\indent \indent a.s.$$for every $l\geq0$. Then the same random variable $x(\omega)$ is called stationary orbit. It is also known as random fixed point.\end{linenomath*}\ed
\bd (\cite{fu2012slow})
\begin{linenomath*} For a random dynamical system $\phi$, a random set $\mathcal{M}=\{\mathcal{M}(\omega):\omega\in \Omega\}$ is said to be random positively invariant set if$$\phi(l,\omega,\mathcal{M}(\omega))\subset \mathcal{M}(\theta_{l}\omega),$$ for every $\omega\in\Omega$ and $l\geq0$.\end{linenomath*}\ed
\bd \cite{yuan2017slow} \begin{linenomath*} Define a map $$h:H_{2}\times\Omega\rightarrow H_{1},$$ such that $y\mapsto h(y,\omega)$ is Lipschitz continuous for every $\omega\in\Omega$. Take $$\mathcal{M}(\omega)=\{(h(y,\omega),y):y\in H_{2}\},$$
such that random positively invariant set $\mathcal{M}=\{\mathcal{M}(\omega):\omega\in \Omega\}$ can be represented as a graph of Lipschitz continuous map $h$, then $\mathcal{M}$ is said to be Lipschitz continuous invariant manifold.\end{linenomath*}\ed
\begin{linenomath*}Moreover, $\mathcal{M}(\omega)$ is said to have exponential tracking property, if there exist an $x'\in \mathcal{M}(\omega)$ for all $x\in \mathbb{H}$ satisfying $$||\phi(l,\omega,x)-\phi(l,\omega,x')||_{\mathbb{H}}\leq c_{1}(x,x',\omega)e^{c_{2}l}||x-x'||_{\mathbb{H}},\indent l\geq0,$$for every $\omega\in \Omega$. Here $c_{1}$ is positive random variable, while  $c_{2}$ is negative constant.
\end{linenomath*}
 \section{Stochastic System to Random Dynamical System}
\begin{linenomath*} In the fast-slow system (1)-(2) processed by symmetric $\alpha$-stable L$\acute{e}$vy noise, the state space for the fast mode is $H_{1}=L^{2}(-1,1)$ and the state space for the slow mode is $H_{2}$. In order to establish the slow manifold, we suppose the following conditions on nonlocal system (1)-(2).\\
\textbf{(S1)} With regards to linear part of (2), there is a constant $\gamma_{J}>0$ such that
$$||e^{Jt}y||_{2}\leq e^{\gamma_{J} t}||y||_{2}, t\leq0, \mbox{ for all } y\in H_{2}.$$
\textbf{(S2)} With regards to nonlinear part of (1)-(2), there is a constant $K>0$ such that for all $(x_{i},y_{i})^{T}$ in $H_{1}\times H_{2}$ and for all
$(x_{j},y_{j})^{T}$ in $H_{1}\times H_{2}$,
$$||f(x_{i},y_{i})-f(x_{j},y_{j})||_{H_{1}}\leq K(||x_{i}-x_{j}||_{H_{1}}+||y_{i}-y_{j}||_{H_{2}}),$$
$$||g(x_{i},y_{i})-g(x_{j},y_{j})||_{H_{2}}\leq K(||x_{i}-x_{j}||_{H_{1}}+||y_{i}-y_{j}||_{H_{2}}),$$
where $T$ indicates the transpose of matrix, and nonlinearities $f$ and $g$
$$f:L^{2} (-1,1)\times H_{2}\rightarrow L^{2} (-1,1),$$
$$g:L^{2} (-1,1)\times H_{2}\rightarrow H_{2},$$
with $f(0,0)=g(0,0)=0$ are $C^{1}$-smooth.\\
\textbf{(S3)} With regards to nonlinear parts of (1)-(2), the Lipschitz constant $K$ is such that
$$K<\frac{\lambda_{1}\gamma_{J}}{\gamma_{J}+2\lambda_{1}}.$$
\indent Now let $\Theta_{1}=(\Omega_{1},\mathcal{F}_{1},\mathbb{P}_{1},\theta_{t}^{1})$ and $\Theta_{2}=(\Omega_{2},\mathcal{F}_{2},\mathbb{P}_{2},\theta_{t}^{2})$ are two independent driving (metric) dynamical system as we explained in Section 2. Define
\begin{align*}\Theta=\Theta_{1}\times\Theta_{2}=(\Omega_{1}\times\Omega_{2},\mathcal{F}_{1}\otimes\mathcal{F}_{2},\mathbb{P}_{1}\times\mathbb{P}_{2},(\theta_{t}^{1},\theta_{t}^{2})^{T}),\end{align*}
and $$\theta_{t}\omega:=(\theta_{t}^{1}\omega_{1},\theta_{t}^{2}\omega_{2})^{T}, \mbox{ for } \omega:=(\omega_{1},\omega_{2})^{T}\in \Omega_{1}\times\Omega_{2}:=\Omega.$$
Let $L_{t}^{\alpha_{1}}$ and $L_{t}^{\alpha_{2}}$ for $\alpha_{1}, \alpha_{2}$ in $(1,2)$ be two mutually independent symmetric $\alpha$-stable L$\acute{e}$vy processes in $H_{1}=L^{2}(-1,1)$ and a separable Hilbert space $H_{2}$ with generating triplet $(a_{1},\mathcal{Q}_{1},v_{1})$ and $(a_{2},\mathcal{Q}_{2},v_{2})$.\\
\indent In order to convert stochastic evolutionary system (1)-(2) into a random system, first we prove the existence and uniqueness of solutions for the stochastic system (1)-(2) and the nonlocal Langevin like equation$$d\eta(t)=A_{\alpha}\eta(t)dt+\sigma dL_{t}^{\alpha}.$$\end{linenomath*}
\begin{lemma} Let $L_{t}^{\alpha}$ be a symmetric $\alpha$-stable L$\acute{e}$vy process, then under supposition (S1-S3), nonlocal system (1)-(2) has a unique solution.\end{lemma}
\begin{proof} Rewrite the system (1)-(2) in the form
\begin{align}\label{sde} \left( {\begin{array}{*{20}{c}}
{\dot{x}}\\
\dot{y}
\end{array}} \right)=
\left( {\begin{array}{*{20}{c}}
{\frac{1}{\epsilon }{A_\alpha }}&0\\
0&J
\end{array}} \right)
\left( {\begin{array}{*{20}{c}}
{x}\\
y
\end{array}} \right)
+\left( {\begin{array}{*{20}{c}}
{\frac{1}{\epsilon}f(x,y)}\\
g(x,y)
\end{array}} \right)
+\left( {\begin{array}{*{20}{c}}
{\frac{\sigma_{1}}{\sqrt[\alpha_{1}]{\epsilon}}\dot{L}_{t}^{\alpha_{1}}}\\
\sigma_{2}\dot{L}_{t}^{\alpha_{2}}
\end{array}} \right)
\end{align}%\nocite{*}
From \cite{bai2017slow}, it is known that $\left( {\begin{array}{*{20}{c}}
{\frac{1}{\epsilon }{A_\alpha }}&0\\
0&J
\end{array}} \right)$ is an infinitesimal generator of a $C_{0}$-semigroup. Then by (\cite{peszat2007stochastic}, p.170), above stochastic evolutionary system has a unique solution.\end{proof}
\begin{lemma}
Let $L_{t}^{\alpha}$ be a symmetric $\alpha$-stable L$\acute{e}$vy process for $\alpha\in (1,2)$ with generating triplet $(a,\mathcal{Q},v)$. Then the nonlocal stochastic equation
\begin{align}
d\eta(t)=A_{\alpha}\eta(t)dt+\sigma dL_{t}^{\alpha},\mbox{ in } L^{2}(-1,1),\end{align}
where $\eta(0)=\eta_{0}$ and $A_{\alpha}$ is the fractional Laplacian operator, posses the solution
\begin{align*}
\eta(t)=e^{-\lambda_{n} t}\eta_{0}+\sigma\int_{0}^{t}e^{-\lambda_{n}(t-s)}dL_{s}^{\alpha}, \mbox{ for } t\geq0,\mbox{ and }n=1,2,3\cdot\cdot\cdot.\end{align*}\end{lemma}
\begin{proof}
From \cite{bostan2013map}, it is known that fractional Laplacian is linear self-adjoint operator. By \cite{kwasnicki2012eigenvalues}, we obtain that there exist an infinite sequence of eigenvalues $\{\lambda_{n}\}$ such that
$$0<\lambda_{1}<\lambda_{2}\leq \lambda_{3}\leq\cdot\cdot\cdot\leq\lambda_{n}\leq \cdot\cdot\cdot, \mbox{ for } n=1,2,3,\cdot\cdot\cdot.$$ and the corresponding eigenfunctions $\varphi_{n}$ form a complete orthonormal set in $L^{2}(-1,1)$ such that $$-(-\Delta)^{\alpha/2}\varphi_{n}=-\lambda_{n}\varphi_{n}.$$ Since $L_{t}^{\alpha}$, $\alpha\in (1,2)$ is a symmetric $\alpha$-stable L$\acute{e}$vy process with exponent $\mathbb{E}e^{i\eta L_{t}^{\alpha}}=e^{-t\psi_{t}(\eta)}$. Here
\begin{align*}\psi_{t}(\eta)=&-i\langle a,\eta\rangle_{L^{2}(-1,1)}+\frac{1}{2}\langle \mathcal{Q}\eta,\eta\rangle_{L^{2}(-1,1)}+\int_{L^{2}(-1,1)}(1-e^{i\langle \eta,y\rangle_{L^{2}(-1,1)}}\\&+i\langle\eta,y\rangle_{L^{2}(-1,1)}c(y))v(dy),\mbox{ with }c(y)=1.\end{align*} From (\cite{sato1999levy}, p.80) it is obtained that
$\alpha\in(1,2)$ if and only if $\int_{|y|>1}|y|v(dy)<\infty$,\\
and by using of (\cite{sato1999levy}, p.163), we get that
$\int_{|y|>1}|y|v(dy)<\infty$ if and only if $L_{t}^{\alpha}$ has finite mean. Finally with  the help of (\cite{sato1999levy}, p.39) we have that
if $\int_{|y|>1}|y|v(dy)<\infty$, then center and mean are identical. Since symmetric $\alpha$-stable L$\acute{e}$vy process for $1<\alpha<2$ has zero mean, so its center $a$ is also zero. Hence
\begin{align*}\psi_{t}(\eta)=&\frac{1}{2}\langle \mathcal{Q}\eta,\eta\rangle_{L^{2}(-1,1)}+\int_{L^{2}(-1,1)}(1-e^{i\langle \eta,y\rangle_{L^{2}(-1,1)}})v(dy).\end{align*}
Then by (\cite{peszat2007stochastic}, p.143) above equation (5) has following solution
\begin{align*}
\eta(t)=e^{-\lambda_{n} t}\eta_{0}+\sigma\int_{0}^{t}e^{-\lambda_{n}(t-s)}dL_{s}^{\alpha}, \mbox{ for } t\geq0,\mbox{ and }n=1,2,3\cdot\cdot\cdot.\end{align*}\end{proof}
\begin{lemma} For a fixed $\epsilon>0$, the equations
\begin{align}
d\eta(t)=\frac{1}{\epsilon}A_{\alpha}\eta(t) dt+\frac{\sigma_{1}}{\sqrt[\alpha_{1}]{\epsilon}}dL_{t}^{\alpha_{1}},\indent \eta(0)=\eta_{0},\end{align}
\begin{align}
d\delta(t)=A_{\alpha}\delta(t) dt+\sigma_{1}dL_{t}^{\alpha_{1}},\indent \delta(0)=\delta_{0},\end{align}
have c$\grave{a}$dl$\grave{a}$g stationary solutions $\sigma_{1}\eta^{\epsilon}(\theta_{t}^{1}\omega_{1})$ and $\sigma_{1}\delta(\theta_{t}^{1}\omega_{1})$ through random variables
$\sigma_{1}\eta^{\epsilon}(\omega_{1})=\frac{\sigma_{1}}{\sqrt[\alpha_{1}]{\epsilon}}\int_{-\infty}^{0}e^{\frac{\lambda_{n}s}{\epsilon}}dL_{s}^{\alpha_{1}}(\omega_{1})$ and
$\sigma_{1}\delta(\omega_{1})=\sigma_{1}\int_{-\infty}^{0}e^{\lambda_{n}s}dL_{s}^{\alpha_{1}}(\omega_{1})$ respectively.\end{lemma}
\begin{proof} The equation (7) has unique c$\grave{a}$dl$\grave{a}$g solution
\begin{align*}
\phi(t,\omega_{1},\delta_{0})=e^{-\lambda_{n}t}\delta_{0}+\sigma_{1}\int_{0}^{t}e^{-\lambda_{n}(t-s)}dL_{s}^{\alpha_{1}}(\omega_{1})\end{align*}
It follows that
\begin{align*}
\phi(t,\omega_{1},\sigma_{1}\delta(\omega_{1}))&=\sigma_{1}e^{-\lambda_{n}t}\delta(\omega_{1})+\sigma_{1}\int_{0}^{t}e^{-\lambda_{n}(t-s)}dL_{s}^{\alpha_{1}}(\omega_{1})
\\&=\sigma_{1}e^{-\lambda_{n}t}\int_{-\infty}^{0}e^{\lambda_{n}s}dL_{s}^{\alpha_{1}}(\omega_{1})+\sigma_{1}\int_{0}^{t}e^{-\lambda_{n}(t-s)}dL_{s}^{\alpha_{1}}(\omega_{1})
\\&=\sigma_{1}\int_{-\infty}^{0}e^{-\lambda_{n}(t-s)}dL_{s}^{\alpha_{1}}(\omega_{1})+\sigma_{1}\int_{0}^{t}e^{-\lambda_{n}(t-s)}dL_{s}^{\alpha_{1}}(\omega_{1})
\\&=\sigma_{1}\int_{-\infty}^{t}e^{-\lambda_{n}(t-s)}dL_{s}^{\alpha_{1}}(\omega_{1}),
\end{align*}
and
\begin{align*}
\sigma_{1}\delta(\theta_{t}^{1}\omega_{1})&=\sigma_{1}\int_{-\infty}^{0}e^{\lambda_{n}s}dL_{s}^{\alpha_{1}}(\theta_{t}^{1}\omega_{1})
\\&=\sigma_{1}\int_{-\infty}^{0}e^{\lambda_{n}s}d(L_{t+s}^{\alpha_{1}}(\omega_{1})-L_{t}^{\alpha_{1}}(\omega_{1}))
\\&=\sigma_{1}\int_{-\infty}^{0}e^{\lambda_{n}s}dL_{t+s}^{\alpha_{1}}(\omega_{1})
=\sigma_{1}\int_{-\infty}^{t}e^{-\lambda_{n}(t-s)}dL_{s}^{\alpha_{1}}(\omega_{1}).
\end{align*}
Hence $\phi(t,\omega_{1},\sigma_{1}\delta(\omega_{1}))=\sigma_{1}\delta(\theta_{t}^{1}\omega_{1})$ is the stationary solution for (7).\\
Similarly (6) has c$\grave{a}$dl$\grave{a}$g stationary solution
\begin{align*}
\sigma_{1}\eta^{\epsilon}(\theta_{t}^{1}\omega_{1})=\frac{\sigma_{1}}{\sqrt[\alpha_{1}]{\epsilon}}\int_{-\infty}^{t}e^{\frac{-\lambda_{n}(t-s)}{\epsilon}}dL_{s}^{\alpha_{1}}(\omega_{1}).
\end{align*}
\end{proof}
\begin{lemma} \cite{yuan2017slow} Similarly the stochastic equation
\begin{align}
d\xi(t)=J\xi(t) dt+\sigma_{2}dL_{t}^{\alpha_{2}},\indent \xi(0)=\xi_{0},\end{align}
has c$\grave{a}$dl$\grave{a}$g stationary solution $\sigma_{2}\xi(\theta_{t}^{2}\omega_{2})$ through random variable $$\sigma_{2}\xi(\omega_{2})=\sigma_{2}\int_{-\infty}^{0}e^{Js}dL_{s}^{\alpha_{2}}(\omega_{2}).$$\end{lemma}
\begin{remark} (\cite{duan2015introduction}, p.191) $L_{ct}^{\alpha}$ and $c^{\frac{1}{\alpha}}L_{t}^{\alpha}$ have the same distribution for every $c>0$, i.e., $$L_{ct}^{\alpha}\overset{d}{=} c^{\frac{1}{\alpha}}L_{t}^{\alpha}, \mbox{ for every }c>0.$$\end{remark}
\begin{lemma} The process $\sigma_{1}\eta^{\epsilon}(\theta_{t}^{1}\omega_{1})$ has the same distribution as the process $\sigma_{1}\delta(\theta_{\frac{t}{\epsilon}}^{1}\omega_{1})$, where $\eta^{\epsilon}$ and $\delta$ are given in previous Lemma 3.3.\end{lemma}
\begin{proof} From Lemma 3.3,
\begin{align*}
\eta^{\epsilon}(\theta_{t}^{1}\omega_{1})&=\frac{1}{\sqrt[\alpha_{1}]{\epsilon}} \int_{-\infty}^{t}e^{\frac{-\lambda_{n}(t-s)}{\epsilon}}dL_{s}^{\alpha_{1}}(\omega_{1})=\int_{-\infty}^{\frac{t}{\epsilon}}e^{-\lambda_{n}(\frac{t}{\epsilon}-r)}\left(\frac{1}{\sqrt[\alpha_{1}]{\epsilon}} dL_{\epsilon r}^{\alpha_{1}}(\omega_{1})\right)\\&\overset{d}{=}\int_{-\infty}^{\frac{t}{\epsilon}}e^{-\lambda_{n}(\frac{t}{\epsilon}-r)} dL_{r}^{\alpha_{1}}(\omega_{1})=\delta(\theta_{\frac{t}{\epsilon}}^{1}\omega_{1}).
\end{align*}
Hence the process $\sigma_{1}\eta^{\epsilon}(\theta_{t}^{1}\omega_{1})$ and the process $\sigma_{1}\delta(\theta_{\frac{t}{\epsilon}}^{1}\omega_{1})$ have the same distribution.\end{proof}
\indent Define a random transformation
 $$\binom{X}{Y}:=\nu(\omega,x,y)=\binom{x-\sigma_{1}\eta^{\epsilon}(\omega_{1})}{{y-\sigma_{2}\xi(\omega_{2})}},$$
 then $(X(t),Y(t))=\nu(\theta_{t}\omega,x,y)$ satisfies the random system
 \begin{align}
 dX&=\frac{1}{\epsilon}A_{\alpha}Xdt+\frac{1}{\epsilon}f(X+\sigma_{1}\eta^{\epsilon}(\theta_{t}^{1}\omega_{1}),Y+\sigma_{2}\xi(\theta_{t}^{2}\omega_{2}))dt,\\
 dY&=JYdt+g(X+\sigma_{1}\eta^{\epsilon}(\theta_{t}^{1}\omega_{1}),Y+\sigma_{2}\xi(\theta_{t}^{2}\omega_{2}))dt.
 \end{align}
 Here the additional terms $\sigma_{1}\eta^{\epsilon}(\theta_{t}^{1}\omega_{1})$ and $\sigma_{2}\xi(\theta_{t}^{2}\omega_{2})$ does not change the Lipschitz constant of nonlinearities $f$ and $g$. So $f$ and $g$ in random dynamical system (9)-(10) and in stochastic dynamical system (1)-(2) have the same Lipschitz constant.
The random system (9)-(10) can be solved for any $\omega\in \Omega$ and for any initial value $(X(0),Y(0))^{T}=(X_{0},Y_{0})^{T}$, then the solution operator
\begin{align*}
(t,\omega,(X_{0},Y_{0})^{T})\mapsto\Phi(t,\omega,(X_{0},Y_{0})^{T})=(X(t,\omega,(X_{0},Y_{0})^{T}),Y(t,\omega,(X_{0},Y_{0})^{T}))^{T},\end{align*}
defines the random dynamical system for (9)-(10). Furthermore,
\begin{align*}
\phi(t,\omega,(X_{0},Y_{0})^{T})=\Phi(t,\omega,(X_{0},Y_{0})^{T})+(\sigma_{1}\eta^{\epsilon}(\theta_{t}^{1}\omega_{1}),\sigma_{2}\xi(\theta_{t}^{2}\omega_{2}))^{T},
\end{align*}
 defines the random dynamical system for (1)-(2).
\section{Random slow manifolds}
We define Banach spaces consist of functions for exploring the random system (9)-(10). For any $\beta \in \tilde{K}\subset\mathbb{R}$:
\begin{align*}
C_{\beta}^{H_{1},-}&=\{\Phi:(-\infty,0]\rightarrow L^{2}(-1,1)\mbox{ is continuous and } \mathop {\mbox{sup} }\limits_{t\in(-\infty,0]}||e^{-\beta t}\Phi(t)||_{1}<\infty\},
\\C_{\beta}^{H_{1},+}&=\{\Phi:[0,\infty)\rightarrow L^{2}(-1,1)\mbox{ is continuous and } \mathop {\mbox{sup} }\limits_{t\in[0,\infty)}||e^{-\beta t}\Phi(t)||_{1}<\infty\},
\end{align*}
having norms
\begin{align*}||\Phi||_{C_{\beta}^{H_{1},-}}=\mathop {\mbox{sup} }\limits_{t\in(-\infty,0]}||e^{-\beta t}\Phi(t)||_{1}, \mbox{ and } ||\Phi||_{C_{\beta}^{H_{1},+}}=\mathop {\mbox{sup} }\limits_{t\in[0,\infty)}||e^{-\beta t}\Phi(t)||_{1}.\end{align*}
Similarly, define
\begin{align*}
C_{\beta}^{H_{2},-}&=\{\Phi:(-\infty,0]\rightarrow H_{2}\mbox{ is continuous and } \mathop {\mbox{sup} }\limits_{t\in(-\infty,0]}||e^{-\beta t}\Phi(t)||_{2}<\infty\},
\\C_{\beta}^{H_{2},+}&=\{\Phi:[0,\infty)\rightarrow H_{2}\mbox{ is continuous and } \mathop {\mbox{sup} }\limits_{t\in[0,\infty)}||e^{-\beta t}\Phi(t)||_{2}<\infty\},
\end{align*}
having norms
\begin{align*}||\Phi||_{C_{\beta}^{H_{2},-}}=\mathop {\mbox{sup} }\limits_{t\in(-\infty,0]}||e^{-\beta t}\Phi(t)||_{2}, \mbox{ and } ||\Phi||_{C_{\beta}^{H_{2},+}}=\mathop {\mbox{sup} }\limits_{t\in[0,\infty)}||e^{-\beta t}\Phi(t)||_{2}.\end{align*}
Let $C_{\beta}^{\pm}$ be the product of Banach spaces $C_{\beta}^{\pm}:=C_{\beta}^{H_{1},\pm}\times C_{\beta}^{H_{2},\pm}$,
having norm
\begin{align*}||Z||_{C_{\beta}^{\pm}}=||X||_{C_{\beta}^{H_{1},\pm}}+||Y||_{C_{\beta}^{H_{2},\pm}},\indent Z=(X,Y)^{T}\in C_{\beta}^{\pm}.\end{align*}
\indent Assume that $0<\gamma<1$ be a number satisfying the property
\begin{align}\label{sde}K<\gamma\lambda_{1}<\lambda_{1} \mbox{ and }-\gamma+\lambda_{1}>K.\end{align}
For convenience, we may consider  $$\gamma=\frac{\gamma_{J}}{2\lambda_{1}+\gamma_{J}}.$$
Let's define
$$\mathcal{M}^{\epsilon}(\omega)\triangleq\{Z_{0}\in \mathbb{H}:Z(t,\omega,Z_{0})\in C_{\beta}^{-}\},\mbox{ with }\beta=-\frac{\gamma}{\epsilon}.$$
Next, we will prove that $\mathcal{M}^{\epsilon}(\omega)$ is an invariant manifold by using of Lyapunov-Perron method.
\begin{lemma} Let $Z(\cdot,\omega)=(X(\cdot,\omega),Y(\cdot,\omega))^{T}$ in $C_{\beta}^{-}$. Then $Z(t,\omega)$ is the solution of (9)-(10) with initial value $Z_{0}=(X_{0},Y_{0})^{T}$ iff $Z(t,\omega)$ satisfies
\begin{align*}&\binom{X(t)}{Y(t)}=\binom{\frac{1}{\epsilon}\int_{-\infty}^{t}e^{A_{\alpha}(t-s)/\epsilon}f(X(s)+\sigma_{1}\eta^{\epsilon}(\theta_{s}^{1}\omega_{1}),Y(s)+\sigma_{2}\xi(\theta_{s}^{2}\omega_{2}))ds}{e^{Jt}Y_{0}+\int_{0}^{t}e^{J(t-s)}g(X(s)+\sigma_{1}\eta^{\epsilon}(\theta_{s}^{1}\omega_{1}),Y(s)+\sigma_{2}\xi(\theta_{s}^{2}\omega_{2}))ds}.\end{align*}\end{lemma}
\begin{proof} If $(X(\cdot,\omega),Y(\cdot,\omega))^{T}$ in $C_{\beta}^{-}$, then by using constants of variation formula, random system (9)-(10) in integral form is
\begin{align}
X(t)&=e^{\frac{A_{\alpha}(t-r)}{\epsilon}}X(r)+\frac{1}{\epsilon}\int_{r}^{t}e^{A_{\alpha}(t-s)/\epsilon}f(X(s)+\sigma_{1}\eta^{\epsilon}(\theta_{s}^{1}\omega_{1}),Y(s)+\sigma_{2}\xi(\theta_{s}^{2}\omega_{2}))ds,\\
Y(t)&=e^{Jt}Y_{0}+\int_{0}^{t}e^{J(t-s)}g(X(s)+\sigma_{1}\eta^{\epsilon}(\theta_{s}^{1}\omega_{1}),Y(s)+\sigma_{2}\xi(\theta_{s}^{2}\omega_{2}))ds.
\end{align}
Since, $(X(\cdot,\omega),Y(\cdot,\omega))^{T}$ in $C_{\beta}^{-}$. So,
\begin{align*}
||e^{\frac{A_{\alpha}(t-r)}{\epsilon}}X(r)||_{C_{\beta}^{H_{1},-}}
&\leq e^{\frac{-\lambda_{1}(t-r)}{\epsilon}}||X(r)||_{C_{\beta}^{H_{1},-}}\\&=e^{\frac{-\lambda_{1}(t-r)}{\epsilon}}\mathop {\mbox{sup} }\limits_{r\in(-\infty,0]}||e^{-\beta r}X(r)||_{1}\\&=
e^{\frac{-\lambda_{1}(t-r)}{\epsilon}}||X(r)||_{1}\rightarrow0, \mbox{ as }r\rightarrow-\infty.\end{align*}
Hence, (12) leads to
\begin{align}
X(t)=\frac{1}{\epsilon}\int_{-\infty}^{t}e^{A_{\alpha}(t-s)/\epsilon}f(X(s)+\sigma_{1}\eta^{\epsilon}(\theta_{s}^{1}\omega_{1}),Y(s)+\sigma_{2}\xi(\theta_{s}^{2}\omega_{2}))ds.
\end{align}
The result follows from (13)-(14).\end{proof}
\begin{lemma}
Suppose that $Z(t,\omega,Z_{0})=(X(t,\omega,(X_{0},Y_{0})^{T} ,Y(t,\omega,(X_{0},Y_{0})^{T}))^{T}$ be the solution of
\begin{align}&\binom{X(t)}{Y(t)}=\binom{\frac{1}{\epsilon}\int_{-\infty}^{t}e^{A_{\alpha}(t-s)/\epsilon}f(X(s)+\sigma_{1}\eta^{\epsilon}(\theta_{s}^{1}\omega_{1}),Y(s)+\sigma_{2}\xi(\theta_{s}^{2}\omega_{2}))ds}{e^{Jt}Y_{0}+\int_{0}^{t}e^{J(t-s)}g(X(s)+\sigma_{1}\eta^{\epsilon}(\theta_{s}^{1}\omega_{1}),Y(s)+\sigma_{2}\xi(\theta_{s}^{2}\omega_{2}))ds},t\leq0.\end{align}
Then $Z^{\epsilon}(t,\omega,Z_{0})$ is the unique solution in $C_{\beta}^{-}$, where $Z_{0}=(X_{0},Y_{0})^{T}$ is the initial value.\end{lemma}
\begin{proof} With the help of Banach fixed point theorem, we prove that $Z(t,\omega,Z_{0})=(X(t,\omega,(X_{0},Y_{0})^{T} ,Y(t,\omega,(X_{0},Y_{0})^{T}))^{T}$ is the unique solution of (15). In order to prove it, let's introduce two operators for $t\leq0$:
$$\mathfrak{K}_{i}(Z)[t]=\frac{1}{\epsilon}\int_{-\infty}^{t}e^{A_{\alpha}(t-s)/\epsilon}f(X(s)+\sigma_{1}\eta^{\epsilon}(\theta_{s}^{1}\omega_{1}),Y(s)+\sigma_{2}\xi(\theta_{s}^{2}\omega_{2}))ds,$$
$$\mathfrak{K}_{j}(Z)[t]=e^{Jt}Y_{0}+\int_{0}^{t}e^{J(t-s)}g(X(s)+\sigma_{1}\eta^{\epsilon}(\theta_{s}^{1}\omega_{1}),Y(s)+\sigma_{2}\xi(\theta_{s}^{2}\omega_{2}))ds.$$
Then Lyapunov-Perron transform is defined to be
$$\mathfrak{K}(Z)=\binom{\mathfrak{K}_{i}(Z)}{\mathfrak{K}_{j}(Z)}=(\mathfrak{K}_{i}(Z),\mathfrak{K}_{j}(Z))^{T}.$$
First we need to prove that the transform $\mathfrak{K}$ maps $C_{\beta}^{-}$ into itself.
For this consider $Z=(X,Y)^{T}$ in $C_{\beta}^{-}$ satisfying:
\begin{align*}||\mathfrak{K}_{i}(Z)[t]||_{C_{\beta}^{H_{1},-}}&=||\frac{1}{\epsilon}\int_{-\infty}^{t}e^{A_{\alpha}(t-s)/\epsilon}f(X(s)+\sigma_{1}\eta^{\epsilon}(\theta_{s}^{1}\omega_{1}),Y(s)+\sigma_{2}\xi(\theta_{s}^{2}\omega_{2}))ds||_{1}\\
&\leq\frac{1}{\epsilon}\mathop { \mbox{sup} }\limits_{t\in (-\infty,0]}\{e^{-\beta (t-s)}\int_{-\infty}^{t}e^{-\lambda_{1}(t-s)/\epsilon}||f(X(s)+\sigma_{1}\eta^{\epsilon}(\theta_{s}^{1}\omega_{1}),Y(s)\\&\indent+\sigma_{2}\xi(\theta_{s}^{2}\omega_{2}))||_{1}ds\\
&\leq\frac{K}{\epsilon}\mathop {\mbox{sup} }\limits_{t\in (-\infty,0]}\{e^{-\beta (t-s)}\int_{-\infty}^{t}e^{-\lambda_{1}(t-s)/\epsilon}(||X(s)||_{1}+||Y(s)||_{2})ds\}+ \mathcal{C}_{i}\\ &\leq\frac{K}{\epsilon}\mathop {\mbox{sup} }\limits_{t\in(-\infty,0]}\{\int_{-\infty}^{t}e^{(-\beta-\lambda_{1}/\epsilon)(t-s)}ds\}||Z||_{C_{\beta}^{-}}+ \mathcal{C}_{i}\\&=\frac{K}{\lambda_{1}+\epsilon\beta}||Z||_{C_{\beta}^{-}}+ \mathcal{C}_{i}.\end{align*}
Similarly, we have
\begin{align*}||\mathfrak{K}_{j}(Z)[t]||_{C_{\beta}^{H_{2},-}}&=||e^{Jt}Y_{0}+\int_{0}^{t}e^{J(t-s)}g(X(s)+\sigma_{1}\eta^{\epsilon}(\theta_{s}^{1}\omega_{1}),Y(s)+\sigma_{2}\xi(\theta_{s}^{2}\omega_{2}))ds||_{2}\\
&\leq \mathop {\mbox{sup} }\limits_{t\in(-\infty,0]}\{e^{-\beta (t-s)}\int_{t}^{0}e^{\gamma_{J}(t-s)}||g(X(s)+\sigma_{1}\eta^{\epsilon}(\theta_{s}^{1}\omega_{1}),Y(s)\\&\indent+\sigma_{2}\xi(\theta_{s}^{2}\omega_{2}))||_{2}ds\}+\mathop {\mbox{sup} }\limits_{t\in (-\infty,0]}\{e^{-\beta t}e^{\gamma_{J} t}||Y_{0}||_{2}\}\\
&\leq K \mathop {\mbox{sup} }\limits_{t\in (-\infty,0]}\{\int_{t}^{0}e^{(\gamma_{J}-\beta)(t-s)}(||X(s)||_{1}+||Y(s)||_{2})ds\}+\mathcal{C}_{j}+||Y_{0}||_{2}\\&=\frac{ K}{-\beta+\gamma_{J}}||Z||_{C_{\beta}^{-}}+\mathcal{C}_{j}+||Y_{0}||_{2}\\&=\frac{K}{ -\beta+\gamma_{J}}||Z||_{C_{\beta}^{-}}+\mathcal{C}_{k}.\end{align*}
By Lyapunov-Perron transform definition $\mathfrak{K}$ in combine form is\begin{align*}||\mathfrak{K}(Z)||_{C_{\beta}^{-}}\leq\varrho(\lambda_{1},\gamma_{J},K,\beta,\epsilon)||Z||_{C_{\beta}^{-}}+\mathcal{C}.\end{align*} Where $\mathcal{C}, \mathcal{C}_{i}, \mathcal{C}_{j}$ and $\mathcal{C}_{k}$ are constants, while \begin{align*}\varrho(\lambda_{1},\gamma_{J},K,\beta,\epsilon)= \frac{K}{\lambda_{1}+\epsilon\beta}+\frac{K}{-\beta+\gamma_{J}}.\end{align*}
Hence $\mathfrak{K}$ maps $C_{\beta}^{-}$ into itself, which means $\mathfrak{K}(Z)$ is in $C_{\beta}^{-}$ for every $Z$ in $C_{\beta}^{-}$.\\
\indent Next, we need to prove that the map $\mathfrak{K}$ is contractive. For this, let's consider $Z=(X,Y)^{T},\bar{Z}=(\tilde{X},\tilde{Y})^{T}\in C_{\beta}^{-}$,
\begin{align*}||\mathfrak{K}_{i}(Z)-\mathfrak{K}_{i}(\tilde{Z})||_{C_{\beta}^{H_{1},-}}&\leq\frac{1}{\epsilon}\mathop {\mbox{sup} }\limits_{t\in (-\infty,0]}\{e^{-\beta (t-s)}\int_{-\infty}^{t}e^{-\lambda_{1}(t-s)/\epsilon}||f(X(s)+\sigma_{1}\eta^{\epsilon}(\theta_{s}^{1}\omega_{1}),Y(s)\\& \indent +\sigma_{2}\xi(\theta_{s}^{2}\omega_{2}))-f(\tilde{X}(s)+\sigma_{1}\eta^{\epsilon}(\theta_{s}^{1}\omega_{1}),\tilde{Y}(s)+\sigma_{2}\xi(\theta_{s}^{2}\omega_{2}))||_{1}ds\}\\
&\leq\frac{K}{\epsilon}\mathop {\mbox{sup} }\limits_{t\in(-\infty,0]}\{e^{-\beta (t-s) }\int_{-\infty}^{t}e^{-\lambda_{1}(t-s)/\epsilon}(||X(s)-\tilde{X}(s)||_{1}\\&\indent+||Y(s)-\tilde{Y}(s)||_{2})ds\\
&\leq\frac{K}{\epsilon}\mathop {\mbox{sup} }\limits_{t\in(-\infty, 0]}\int_{-\infty}^{t}e^{(\frac{-\lambda_{1}}{\epsilon}-\beta)(t-s)}ds\}||Z-\tilde{Z}||_{C_{\beta}^{-}}\\
&=\frac{K}{\lambda_{1}+\epsilon\beta}||Z-\tilde{Z}||_{C_{\beta}^{-}}.\end{align*}
Using the same way
\begin{align*}||\mathfrak{K}_{j}(Z)-\mathfrak{K}_{j}(\tilde{Z})||_{C_{\beta}^{H_{2},-}}&\leq K \mathop {\mbox{sup} }\limits_{t\in(-\infty, 0]}\{\int_{t}^{0}e^{\gamma_{J}(t-s)}e^{-\beta(t-s)}ds\}||Z-\tilde{ Z}||_{C_{\beta}^{-}}\\
&\leq K \mathop {\mbox{sup} }\limits_{t\in (-\infty, 0]}\int_{t}^{0}e^{(\gamma_{J}-\beta)(t-s)}ds\}||Z-\tilde{Z}||_{C_{\beta}^{-}}\\&=\frac{ K}{-\beta+\gamma_{J}}||Z-\tilde{Z}||_{C_{\beta}^{-}}.\end{align*}
In combine form\begin{align*} &||\mathfrak{K}(Z)-\mathfrak{K}(\tilde{Z})||_{C_{\beta}^{-}}\leq\varrho(\lambda_{1},\gamma_{J},K,\beta,\epsilon)||Z-\tilde{Z}||_{C_{\beta}^{-}},\end{align*}
where\begin{align*}&\varrho(\lambda_{1},\gamma_{J},K,\beta,\epsilon)=\frac{K}{\lambda_{1}+\epsilon\beta}+\frac{K}{-\beta+\gamma_{J}}\rightarrow\frac{K}{\lambda_{1}}+\frac{K}{-\beta+\gamma_{J}} \mbox{ for } \epsilon\rightarrow0.\end{align*}
By the supposition (S3), and $\beta=-\frac{\gamma}{\epsilon}$,
\begin{align*}\varrho(\lambda_{1},\gamma_{J},K,\beta,\epsilon)\rightarrow\frac{K}{\lambda_{1}} \mbox{ for } \epsilon\rightarrow0.\end{align*}
So, there is a very small parameter $\epsilon_{0}\rightarrow 0$ such that
\begin{align*}0<\varrho(\lambda_{1},\gamma_{J},K,\beta,\epsilon)< 1,\mbox{ for }\epsilon \mbox{ in } (0,\epsilon_{0}).\end{align*} Hence, by definition of contractive mapping, the map $\mathfrak{K}$ is contractive in $C_{-\frac{\gamma}{\epsilon}}^{-}$. By Banach fixed point theorem, every contractive mapping in non-empty Banach space has a unique fixed point, which is a unique solution. Hence (15) has the unique solution \begin{align*}Z(t,\omega,Z_{0})=(X(t,\omega,(X_{0},Y_{0})^{T}),Y(t,\omega,(X_{0},Y_{0})^{T}))^{T} \mbox{ in } C_{-\frac{\gamma}{\epsilon}}^{-}.\end{align*}\end{proof}
From Lemma 4.2 we get the following remark.\\
\begin{remark} For any $(X_{0},Y_{0})^{T}$, $(X'_{0},Y'_{0})^{T}$ in $\mathbb{H}$, and for all $\omega \in \Omega, Y_{0}, Y'_{0} \in H_{2},$ there is an $\epsilon_{0}>0$ such that
\begin{align}||Z(t,\omega,(X_{0},Y_{0})^{T})-Z(t,\omega,(X'_{0},Y'_{0})^{T})||_{C_{-\frac{\gamma}{\epsilon}}^{-}}\leq\frac{1}{1-\varrho(\lambda_{1},\gamma_{J},K,\beta,\epsilon)}||Y_{0}-Y_{0}'||_{2}.\end{align}\end{remark}
\begin{proof} For the sake of simplicity, instead of writing $Z(t,\omega,(X_{0},Y_{0})^{T})$ and $Z(t,\omega,(X'_{0},Y'_{0})^{T})$, let's write $Z(t,\omega,Y_{0})$ and $Z(t,\omega,Y'_{0})$. For all $\omega \in \Omega$ and $Y_{0}, Y'_{0}$ in $H_{2}$, we have the upper-bound
 \begin{align*}||Z(t,\omega,Y_{0})-Z(t,\omega,Y'_{0})||_{C_{-\frac{\gamma}{\epsilon}}^{-}} &=||X(t,\omega,Y_{0})-X(t,\omega,Y'_{0})||_{C_{-\frac{\gamma}{\epsilon}}^{H_{1},-}}+||Y(t,\omega,Y_{0})\\&\indent-Y(t,\omega,Y'_{0})||_{C_{-\frac{\gamma}{\epsilon}}^{H_{2},-}}\\
&\leq\frac{K}{\lambda_{1}+\epsilon\beta}||Z(t,\omega,Y_{0})-Z(t,\omega,Y'_{0})||_{C_{-\frac{\gamma}{\epsilon}}^{-}}+\frac{ K}{-\beta+\gamma_{J}}\\&\indent\times||Z(t,\omega,Y_{0})-Z(t,\omega,Y'_{0})||_{C_{-\frac{\gamma}{\epsilon}}^{-}} +||Y_{0}-Y'_{0}||_{2}\\
&=\varrho(\lambda_{1},\gamma_{J},K,\beta,\epsilon)||Z(t,\omega,Y_{0})-Z(t,\omega,Y'_{0})||_{C_{-\frac{\gamma}{\epsilon}}^{-}}\\&\indent+||Y_{0}-Y'_{0}||_{2}.\end{align*}
Thus, \small{\begin{align}||Z(t,\omega,(X_{0},Y_{0})^{T})-Z(t,\omega,(X'_{0},Y'_{0})^{T})||_{C_{-\frac{\gamma}{\epsilon}}^{-}}\leq\frac{1}{1-\varrho(\lambda_{1},\gamma_{J},K,\beta,\epsilon)}||Y_{0}-Y_{0}'||_{2}.\end{align}}
\end{proof}
\begin{theorem} Let suppositions (S1-S3) satisfied. Then for sufficiently small $\epsilon>0$, random system of equations(9)-(10) posses a Lipschitz random slow manifold:$$\mathcal{M}^{\epsilon}(\omega)=\{(\mathcal{H}^{\epsilon}(\omega,Y_{0}),Y_{0})^{T}:Y_{0}\in H_{2}\},$$where $$\mathcal{H}^{\epsilon}(\cdot,\cdot):\Omega\times H_{2}\rightarrow L^{2}(-1,1),$$ is a Lipschitz continuous graph map having Lipschitz constant$$Lip \mathcal{H}^{\epsilon}(\omega,\cdot)\leq\frac{K}{(\lambda_{1}-\gamma)[1-K(\frac{1}{\lambda_{1}-\gamma}+\frac{\epsilon}{\gamma+\epsilon\gamma_{J}})]}.$$\end{theorem}
\begin{proof} For any $Y_{0}\in H_{2}$, introduce the Lyapunov-Perron map  $\mathcal{H}^{\epsilon}:$\begin{equation}\label{sde}\mathcal{H}^{\epsilon}(\omega, Y_{0})=\frac{1}{\epsilon}\int_{-\infty}^{0}e^{-A_{\alpha}s/\epsilon}f(X(s,\omega,Y_{0})+\sigma_{1}\eta^{\epsilon}(\theta_{s}^{1}\omega_{1}),Y(s,\omega,Y_{0})+\sigma_{2}\xi(\theta_{s}^{2}\omega_{2}))ds,\end{equation}
then by (17), the following upper-bound is obtained $$||\mathcal{H}^{\epsilon}(\omega,Y_{0})-H^{\epsilon}(\omega,Y'_{0})||_{1}\leq\frac{K}{\lambda_{1}+\epsilon\beta}\frac{1}{[1-\varrho(\lambda_{1},\gamma_{J},K,\beta,\epsilon)]}||Y_{0}-Y'_{0}||_{2},$$
for all $Y_{0},Y'_{0} \in H_{2}$ and $\omega \in \Omega$. So$$||\mathcal{H}^{\epsilon}(\omega,Y_{0})-\mathcal{H}^{\epsilon}(\omega,Y'_{0})||_{1}\leq\frac{K}{-\gamma+\lambda_{1}}\frac{1}{[1-\varrho(\lambda_{1},\gamma_{J},K,\beta,\epsilon)]}||Y_{0}-Y'_{0}||_{2},$$
for every $Y_{0},Y'_{0} \in H_{2}$ and $\omega \in \Omega$. Then by Lemma 4.1, \label{sde} $$\mathcal{M}^{\epsilon}(\omega)=\{(\mathcal{H}^{\epsilon}(\omega,Y_{0}),Y_{0})^{T}:Y_{0}\in H_{2}\}.$$
Next by using of Theorem III.9 in Casting and Valadier (\cite{castaing1977convex}, p.67), $\mathcal{M}^{\epsilon}(\omega)$ is a random set, i.e., for any $Z=(X,Y)^{T}$ in $\mathbb{H}=H_{1}\times H_{2}$,\begin{equation}\label{sde}\omega\mapsto \mathop {\mbox{inf} }\limits_{Z'\in \mathbb{H}}||(X,Y)^{T}-(\mathcal{H}^{\epsilon}(\omega,\mathfrak{K}Z'),\mathfrak{K}Z')^{T}||,\end{equation}is measurable. Let there is a countable dense set, say, $\mathbb{H}_{c}$ of separable space $\mathbb{H}$. Then right side of (19) is \begin{equation}\label{sde}\mathop {\mbox{inf} }\limits_{Z'\in \mathbb{H}_{c}}||(X,Y)^{T}-(\mathcal{H}^{\epsilon}(\omega,\mathfrak{K}Z'),\mathfrak{K}Z')^{T}||.\end{equation} Under infimum of (19) the measurability of any expression can be obtained, since $\omega\mapsto \mathcal{H}^{\epsilon}(\omega,\mathfrak{K}Z')$ is measurable for all $Z'$ in $\mathbb{H}.$\\
\indent Now it remains to prove that $\mathcal{M}^{\epsilon}(\omega)$ is positively invariant in the sense: for all $Z_{0} =(X_{0},Y_{0})^{T}$ in $\mathcal{M}^{\epsilon}(\omega),$ $ Z(s,\omega,Z_{0})$ is in $\mathcal{M}^{\epsilon}(\theta_{s}\omega)$ for each $s\geq 0.$ Observe that $Z(t+s,\omega,Z_{0})$ is a solution of \begin{align*}
&dX=\frac{1}{\epsilon}A_{\alpha}Xdt+\frac{1}{\epsilon}f(X+\sigma_{1}\eta^{\epsilon}(\theta_{t}^{1}\omega_{1}), Y+\sigma_{2}\xi(\theta_{t}^{2}\omega_{2}))dt,\\&dY=JYdt+g(X+\sigma_{1}\eta^{\epsilon}(\theta_{t}^{1}\omega_{1}), Y+\sigma_{2}\xi(\theta_{t}^{2}\omega_{2}))dt,\end{align*} with initial value $Z(0)=(X(0),Y(0))^{T}=Z(s,\omega,Z_{0})$. So, $Z(t+s,\omega,Z_{0})=Z(t,\theta_{s}\omega,Z(s,\omega,Z_{0}))$. Since $Z(t,\omega,Z_{0})$ in $C_{-\frac{\gamma}{\epsilon}}^{-}$, then
$Z(t,\theta_{s}\omega,Z(s,\omega,Z_{0}))$ in $C_{-\frac{\gamma}{\epsilon}}^{-}$. Hence, $Z(s,\omega,Z_{0}) \in \mathcal{M}^{\epsilon}(\theta_{s}\omega)$. It completes the proof.\end{proof}
\begin{theorem}
Let suppositions (S1-S3) satisfied. Then for sufficiently small $\epsilon>0$, random invariant manifold of random system (9)-(10) posses the exponential tracking property: there exist $\check{Z}_{0}=(\check{X}_{0},\check{Y}_{0})^{T}\in \mathcal{M}^{\epsilon}(\omega)$, for all $Z_{0}=(X_{0},Y_{0})^{T} \in \mathbb{H}$, such that
$$||\Phi(t,\omega,Z_{0})-\check{\Phi}(t,\omega,\check{Z}_{0})||\leq\mathcal{C}_{i}e^{-\mathcal{C}_{j}t}||Z_{0}-\check{Z}_{0}||,\indent t\geq0.$$
Where $\mathcal{C}_{i}$ and $\mathcal{C}_{j}$ are positive constants.\end{theorem}
\begin{proof}
Assume that there are two dynamical orbits for random system (9)-(10), i.e., \begin{align*}\Phi(t,\omega,Z_{0})=(X(t,\omega,Z_{0}), Y(t,\omega,{Z}_{0}))^{T}\end{align*} and \begin{align*}\check{\Phi}(t,\omega,\check{Z}_{0})=(X(t,\omega,\check{Z}_{0}), Y(t,\omega,\check{Z}_{0}))^{T}.\end{align*} Then the difference \begin{align*}\Psi(t)=\check{\Phi}(t,\omega,\check{Z}_{0})-\Phi(t,\omega,Z_{0}):=(U(t),V(t))^{T}\end{align*} satisfies the equations
\begin{align}&dU=\frac{1}{\epsilon}A_{\alpha}Udt+\frac{1}{\epsilon}\mathrm{\tilde{F}}(U,V,\sigma_{1}\eta^{\epsilon}(\theta_{t}^{1}\omega_{1}),\sigma_{2}\xi(\theta_{t}^{2}\omega_{2}))dt,\\
&dV=JVdt+\mathrm{\tilde{G}}(U,V,\sigma_{1}\eta^{\epsilon}(\theta_{t}^{1}\omega_{1}),\sigma_{2}\xi(\theta_{t}^{2}\omega_{2}))dt.\end{align}
Where nonlinearities $\mathrm{\tilde{F}}$ and $\mathrm{\tilde{G}}$ are
\begin{align*}\mathrm{\tilde{F}}(U,V,\sigma_{1}\eta^{\epsilon}(\theta_{t}^{1}\omega_{1}),\sigma_{2}\xi(\theta_{t}^{2}\omega_{2}))
=&f(U+X+\sigma_{1}\eta^{\epsilon}(\theta_{t}^{1}\omega_{1}),V+Y+\sigma_{2}\xi(\theta_{t}^{2}\omega_{2}))\\&-f(X+\sigma_{1}\eta^{\epsilon}(\theta_{t}^{1}\omega_{1}),Y+\sigma_{2}\xi(\theta_{t}^{2}\omega_{2})),\end{align*}
\begin{align*}\mathrm{\tilde{G}}(U,V,\sigma_{1}\eta^{\epsilon}(\theta_{t}^{1}\omega_{1}),\sigma_{2}\xi(\theta_{t}^{2}\omega_{2}))
=&g(U+X+\sigma_{1}\eta^{\epsilon}(\theta_{t}^{1}\omega_{1}),V+Y+\sigma_{2}\xi(\theta_{t}^{2}\omega_{2}))\\&-g(X+\sigma_{1}\eta^{\epsilon}(\theta_{t}^{1}\omega_{1}),Y+\sigma_{2}\xi(\theta_{t}^{2}\omega_{2})).\end{align*}
\indent First, we claim that $\Psi(t)=(U(t),V(t))^{T}$ is a solution of (21)-(22) in $C_{\beta}^{+}$ for $\beta=-\frac{\gamma}{\epsilon}$
if
\begin{equation}\label{sde}\binom{U(t)}{V(t)}=\binom{e^{A_{\alpha}t/\epsilon}U(0)+\frac{1}{\epsilon}\int_{0}^{t}e^{A_{\alpha}(t-s)/\epsilon}\mathrm{\tilde{F}}(U,V,\sigma_{1}\eta^{\epsilon}(\theta_{s}^{1}\omega_{1}),\sigma_{2}\xi(\theta_{s}^{2}\omega_{2}))ds}{\int_{+\infty}^{t}e^{J(t-s)}\mathrm{\tilde{G}}(U,V,\sigma_{1}\eta^{\epsilon}(\theta_{s}^{1}\omega_{1}),\sigma_{2}\xi(\theta_{s}^{2}\omega_{2}))ds}.\end{equation}
It is proved with the help of variation of constants formula just like Lemma 4.1. Since the steps of proof are similar as in Lemma 4.1, so here we omit the proof. Next, it need to prove that $(U,V)^{T}$ is unique solution of (23) in $C_{\beta}^{+}$ with initial value $(U(0),V(0))^{T}=(U_{0},V_{0})^{T}$ such that $$(\check{X}_{0},\check{Y}_{0})^{T}=(U_{0},V_{0})^{T}+(X_{0},Y_{0})^{T}\in \mathcal{M}^{\epsilon}(\omega).$$
It is clear that
$$(\check{X}_{0},\check{Y}_{0})^{T}\in \mathcal{M}^{\epsilon}(\omega)$$ if and only if $$\check{X}_{0}=\frac{1}{\epsilon}
\int_{ - \infty }^{0} e^{A_{\alpha}(-s)}\mathrm{\tilde{F}}(X(s,\check{Y}_{0}),Y(s,\check{Y}_{0}),\sigma_{1}\eta^{\epsilon}(\theta_{s}^{1}\omega_{1}),\sigma_{2}\xi(\theta_{s}^{2}\omega_{2}))ds.$$
Since here $$(\check{X}_{0},\check{Y}_{0})^{T}=(U_{0},V_{0})^{T}+(X_{0},Y_{0})^{T}.$$ So it follows that
$$(\check{X}_{0},\check{Y}_{0})^{T}=(U_{0},V_{0})^{T}+(X_{0},Y_{0})^{T}\in \mathcal{M}^{\epsilon}(\omega)$$ if and only if \small{\begin{align*}U_{0}+X_{0}=&\frac{1}{\epsilon}\int_{-\infty}^{0}e^{A_{\alpha}(-s)}\mathrm{\tilde{F}}(X(s,{V}_{0}+Y_{0}),Y(s,{V}_{0}+Y_{0}),\sigma_{1}\eta^{\epsilon}(\theta_{s}^{1}\omega_{1}),\sigma_{2}\xi(\theta_{s}^{2}\omega_{2}))ds
\\=&\mathcal{H}^{\epsilon}(\omega,V_{0}+Y_{0}).\end{align*}} In short $$(\check{X}_{0},\check{Y}_{0})^{T}=(U_{0},V_{0})^{T}+(X_{0},Y_{0})^{T}\in \mathcal{M}^{\epsilon}(\omega)\triangleq\{Z_{0}\in \mathbb{H}:Z(t,\omega,Z_{0})\in C_{\beta}^{+}\},$$ if and only if \begin{align}U_{0}=-X_{0}+\mathcal{H}^{\epsilon}(\omega,V_{0}+Y_{0}).\end{align}
 For every $\Psi=(U,V)^{T}\in C_{\beta}^{+}$, take $\beta=-\frac{\gamma}{\epsilon}$, $t\geq0$ and define two operators
\begin{align*} &\mathfrak{J}_{i}(\Psi)[t]:=e^{A_{\alpha}t/\epsilon}U_{0}+\frac{1}{\epsilon}\int_{0}^{t}e^{A_{\alpha}(t-s)/\epsilon}\mathrm{\tilde{F}}(U(s),V(s),\sigma_{1}\eta^{\epsilon}(\theta_{s}^{1}\omega_{1}),\sigma_{2}\xi(\theta_{s}^{2}\omega_{2}))ds,\\
&\mathfrak{J}_{j}(\Psi)[t]:=\int_{+\infty}^{t}e^{F(t-s)}\mathrm{\tilde{G}}(U(s),V(s),\sigma_{1}\eta^{\epsilon}(\theta_{s}^{1}\omega_{1}),\sigma_{2}\xi(\theta_{s}^{2}\omega_{2}))ds.\end{align*}
Furthermore, Lyapunov-Perron transform $\mathfrak{J}:C_{-\frac{\gamma}{\epsilon}}^{+}\rightarrow C_{-\frac{\gamma}{\epsilon}}^{+}$ is defined as:
\begin{align*}\mathfrak{J}(\Psi)[t]=\binom{\mathfrak{J}_{i}(\Psi)[t]}{\mathfrak{J}_{j}(\Psi)[t]}=(\mathfrak{J}_{i}(\Psi)[t],\mathfrak{J}_{j}(\Psi)[t])^{T}.\end{align*}
For any $\Psi=(U,V)^{T},\check{\Psi}=(\check{U},\check{V})^{T}\in C_{-\frac{\gamma}{\epsilon}}^{+},$ we obtain the estimate from (24)
\begin{align*}||e^{A_{\alpha}t/\epsilon}(U_{0}-\check{U}_{0})||_{1}&\leq e^{-\lambda_{1}t/\epsilon}LipH^{\epsilon}||V_{0}-\check{V}_{0}||_{2}\\&\leq e^{-\lambda_{1}t/\epsilon}LipH^{\epsilon}||\int_{+\infty}^{0}e^{J(-s)}(\mathrm{\tilde{G}}(\Psi(s),\sigma_{1}\eta^{\epsilon}(\theta_{s}^{1}\omega_{1}),\sigma_{2}\xi(\theta_{s}^{2}\omega_{2}))\\&\indent-\mathrm{\tilde{G}}(\check{\Psi}(s),\sigma_{1}\eta^{\epsilon}(\theta_{s}^{1}\omega_{1}),\sigma_{2}\xi(\theta_{s}^{2}\omega_{2})))ds||_{2}
\\&\leq e^{-\lambda_{1}t/\epsilon}LipH^{\epsilon}K\int_{0}^{+\infty}e^{-\gamma_{J} s}||\Psi(s)-\check{\Psi}(s)||ds.
\end{align*}
So,
\begin{align*}||\mathfrak{J}_{i}(\Psi-\check{\Psi})||_{C_{\beta}^{1,+}}&\leq LipH^{\epsilon}\times K||\Psi-\check{\Psi}||_{C_{\beta}^{+}}\mathop {\mbox{sup} }\limits_{t\in[0,\infty)}\{e^{-(\beta+\frac{\lambda_{1}}{\epsilon})t}\int_{0}^{+\infty}e^{(-\gamma_{J}+\beta)s}ds\}\\&\indent+\frac{K}{\epsilon}||\Psi-\check{\Psi}||_{C_{\beta}^{+}}
\mathop {\mbox{sup} }\limits_{t\in[0,\infty)}\{e^{-\beta t}\int_{0}^{t}e^{-\lambda_{1}(t-s)/\epsilon}ds\}.\end{align*}
Hence
\begin{align}\label{sde}||\mathfrak{J}_{i}(\Psi-\check{\Psi})||_{C_{\beta}^{1,+}}\leq(\frac{Lip H^{\epsilon}\times K}{-\beta+\gamma_{J}}+\frac{K}{\lambda_{1}+\epsilon\beta})||\Psi-\check{\Psi}||_{C_{\beta}^{+}}.\end{align}
By the same way
\begin{align*}||\mathfrak{J}_{j}(\Psi-\check{\Psi})||_{C_{\beta}^{2,+}}\leq K||\Psi-\check{\Psi}||_{C_{\beta}^{+}}\mathop {\mbox{sup} }\limits_{t\in[0,\infty)}\{e^{-\beta t}\int_{t}^{+\infty}e^{(\gamma_{J})(t-s)}ds\}.\end{align*}
This implies \begin{align}\label{sde} ||\mathfrak{J}_{j}(\Psi-\check{\Psi})||_{C_{\beta}^{2,+}}\leq\frac{ K}{-\beta+\gamma_{J}}||\Psi-\check{\Psi}||_{C_{\beta}^{+}}.\end{align}
From Theorem 4.4, it is known that
\begin{align*}Lip\mathcal{H}^{\epsilon}(\omega,.)\leq\frac{K}{(\lambda_{1}-\gamma)[1-K(\frac{1}{\lambda_{1}-\gamma}+\frac{\epsilon}{\gamma+\epsilon\gamma_{J}})]}.\end{align*}
Now, (25)-(26)in combine form is obtained as
\begin{align*}||\mathfrak{J}(\Psi-\check{\Psi})||_{C_{-\frac{\gamma}{\epsilon}}^{+}}\leq\rho(\lambda_{1},\gamma_{J},K,\gamma,\epsilon)||\Psi-\check{\Psi}||_{C_{-\frac{\gamma}{\epsilon}}^{+}},\end{align*}
where,
\begin{align*}\rho(\lambda_{1},\gamma_{J},K,\gamma,\epsilon)&=\frac{K}{\lambda_{1}+\epsilon\beta}+\frac{ K}
{-\beta+\gamma_{J}}+\frac{ K^{2}}{(\lambda_{1}-\gamma)(-\beta+\gamma_{J})[1-K(\frac{1}{\lambda_{1}-\gamma}+\frac{\epsilon}{\gamma+\epsilon\gamma_{J}})]},\\
&\rightarrow\frac{K}{\lambda_{1}}+\frac{K}
{-\beta+\gamma_{J}}+\frac{ K^{2}}{(\lambda_{1}-\gamma)(-\beta+\gamma_{J})[1-K(\frac{1}{\lambda_{1}-\gamma})]},\mbox{ as } \epsilon\rightarrow0.\end{align*}
By taking $\beta=-\frac{\gamma}{\epsilon}$, it is obtained that
\begin{align}\rho(\lambda_{1},\gamma_{J},K,\gamma,\epsilon)\rightarrow\frac{K}{\lambda_{1}}\mbox{ as }\epsilon\rightarrow0.\end{align}
By (11), there is a sufficiently small constant $\check{\epsilon}_{0}>0$ such that
\begin{align*}\rho(\lambda_{1},\gamma_{J},K,\gamma,\epsilon)<1,\mbox{ for all }0<\epsilon<\check{\epsilon}_{0}.\end{align*}
So, the operator $\mathfrak{J}$ is strictly contractive and has a unique fixed point $\Psi$ in $C_{-\frac{\gamma}{\epsilon}}^{+}$. By Banach fixed point theorem, this unique fixed point is called unique solution of (23) and it satisfies
\begin{align*}(\check{X}_{0},\check{Y}_{0})^{T}=(U_{0},V_{0})^{T}+(X_{0},Y_{0})^{T}\in \mathcal{M}^{\epsilon}(\omega).\end{align*}
Furthermore, we have
\begin{align*}||\Psi||_{C_{-\frac{\gamma}{\epsilon}}^{+}}\leq\frac{1}{1-K(\frac{1}{\lambda_{1}-\gamma}+\frac{\epsilon}{\gamma+\epsilon\gamma_{J}})}||\Psi_{0}||_{C_{-\frac{\gamma}{\epsilon}}^{+}},\end{align*}
this implies that
\begin{align*}||\Phi(t,\omega,Z_{0})-\check{\Phi}(t,\omega,\check{Z}_{0})||_{C_{-\frac{\gamma}{\epsilon}}^{+}}\leq \frac{e^{-\frac{\gamma}{\epsilon}t}}{1-K(\frac{1}{\lambda_{1}-\gamma}+\frac{\epsilon}{\gamma+\epsilon\gamma_{J}})}||Z_{0}-\check{Z}_{0}||_{C_{-\frac{\gamma}{\epsilon}}^{+}},t\geq0.\end{align*}
Hence, it obtains the exponential tracking property of $\mathcal{M}^{\epsilon}(\omega)$.\end{proof}
\begin{remark}
From Theorem 4.4 and Theorem 4.5, it is concluded that the random dynamical system has an exponential tracking random slow manifold. Since there is a relation between solutions of stochastic system (1)-(2) and random system (9)-(10). So if (1)-(2) satisfies the suppositions of Theorem 4.4 and Theorem 4.5, then it also posses exponential tracking random slow manifold, i.e.,
\begin{align*}
\mathcal{\tilde{M}}^{\epsilon}(\omega)=\mathcal{M}^{\epsilon}(\omega)+(\sigma_{1}\eta^{\epsilon}(\omega_{1}),\sigma_{2}\xi(\omega_{2}))^{T}=\{(\mathcal{\tilde{H}}^{\epsilon}(\omega,Y_{0}),Y_{0})^{T}:Y_{0}\in H_{2}\},\end{align*}
where,
\begin{align*}
\mathcal{\tilde{H}}^{\epsilon}(\omega,Y_{0})=\mathcal{H}^{\epsilon}(\omega,Y_{0})+\sigma_{1}\eta^{\epsilon}(\omega_{1}).\end{align*}
\end{remark}
\section{Approximation of a random slow manifold}
From random system (9)-(10), we get the following equations by letting time scale $\tau=\frac{t}{\epsilon}$,
\begin{align}
\frac{dX(\tau\epsilon)}{d\tau}&=A_{\alpha}X(\tau\epsilon)+f(X(\tau\epsilon)+\sigma_{1}\eta^{\epsilon}(\theta_{\tau\epsilon}^{1}\omega_{1}),Y(\tau\epsilon)+\sigma_{2}\xi(\theta_{\tau\epsilon}^{2}\omega_{2})),\\
\frac{dY(\tau\epsilon)}{d\tau}&=\epsilon[JY(\tau\epsilon)+g(X(\tau\epsilon)+\sigma_{1}\eta^{\epsilon}(\theta_{\tau\epsilon}^{1}\omega_{1}),Y(\tau\epsilon)+\sigma_{2}\xi(\theta_{\tau\epsilon}^{2}\omega_{2}))].\end{align}
In integral form (28)-(29) can be written as
\begin{align}
X(\tau\epsilon)&=\int_{-\infty}^{\tau}e^{A_{\alpha}(\tau-s)}f(X(s\epsilon)+\sigma_{1}\eta^{\epsilon}(\theta_{s\epsilon}^{1}\omega_{1}),Y(s\epsilon)+\sigma_{2}\xi(\theta_{s\epsilon}^{2}\omega_{2}))ds,\\
Y(\tau\epsilon)&=Y_{0}+\epsilon\int_{0}^{\tau}[JY(s\epsilon)+g(X(s\epsilon)+\sigma_{1}\eta^{\epsilon}(\theta_{s\epsilon}^{1}\omega_{1}),Y(s\epsilon)+\sigma_{2}\xi(\theta_{s\epsilon}^{2}\omega_{2}))]ds.\end{align}
For a sufficiently small $\epsilon>0$, we approximate the slow manifold by expanding the solution of (28) such as
\begin{align}
X(\tau\epsilon)=X_{0}(\tau)+\epsilon X_{1}(\tau)+\epsilon^{2}X_{2}(\tau)+\cdot\cdot\cdot,\end{align}
with initial data
\begin{align}
X(0)=\mathcal{H}^{\epsilon}(\omega,Y_{0})=\mathcal{H}^{(0)}(\omega,Y_{0})+\epsilon\mathcal{H}^{(1)}(\omega,Y_{0})+\epsilon^{2}\mathcal{H}^{2}(\omega,Y_{0})+\cdot\cdot\cdot.\end{align}
We have the Taylor expansions
\begin{align*}
f&(X(\tau\epsilon)+\sigma_{1}\eta^{\epsilon}(\theta_{\tau\epsilon}^{1}\omega_{1}),Y(\tau\epsilon)+\sigma_{2}\xi(\theta_{\tau\epsilon}^{2}\omega_{2}))
\\=&f(X_{0}+\sigma_{1}\eta^{\epsilon}(\theta_{\tau\epsilon}^{1}\omega_{1}),Y_{0}+\sigma_{2}\xi(\theta_{\tau\epsilon}^{2}\omega_{2}))+f_{X}(X_{0}+\sigma_{1}\eta^{\epsilon}(\theta_{\tau\epsilon}^{1}\omega_{1}),Y_{0}+\sigma_{2}\xi(\theta_{\tau\epsilon}^{2}\omega_{2}))\\&(X(\tau\epsilon)-X_{0})+f_{Y}(X_{0}+\sigma_{1}\eta^{\epsilon}(\theta_{\tau\epsilon}^{1}\omega_{1}),Y_{0}+\sigma_{2}\xi(\theta_{\tau\epsilon}^{2}\omega_{2}))(Y(\tau\epsilon)-Y_{0}),\\=&f(X_{0}+\sigma_{1}\eta^{\epsilon}(\theta_{\tau\epsilon}^{1}\omega_{1}),Y_{0}+\sigma_{2}\xi(\theta_{\tau\epsilon}^{2}\omega_{2}))+f_{X}(X_{0}+\sigma_{1}\eta^{\epsilon}(\theta_{\tau\epsilon}^{1}\omega_{1}),Y_{0}+\sigma_{2}\xi(\theta_{\tau\epsilon}^{2}\omega_{2}))\\&\[\epsilon X_{1}(\tau)+\epsilon^{2}X_{2}(\tau)+\cdot\cdot\cdot\]+f_{Y}(X_{0}+\sigma_{1}\eta^{\epsilon}(\theta_{\tau\epsilon}^{1}\omega_{1}),Y_{0}+\sigma_{2}\xi(\theta_{\tau\epsilon}^{2}\omega_{2}))\[\epsilon\int_{0}^{\tau}[JY(s\epsilon)\\&+g(X(s\epsilon)+\sigma_{1}\eta^{\epsilon}(\theta_{s\epsilon}^{1}\omega_{1}),Y(s\epsilon)+\sigma_{2}\xi(\theta_{s\epsilon}^{2}\omega_{2}))]ds\],\end{align*}
and
\begin{align*}
g&(X(\tau\epsilon)+\sigma_{1}\eta^{\epsilon}(\theta_{\tau\epsilon}^{1}\omega_{1}),Y(\tau\epsilon)+\sigma_{2}\xi(\theta_{\tau\epsilon}^{2}\omega_{2}))
\\=&g(X_{0}+\sigma_{1}\eta^{\epsilon}(\theta_{\tau\epsilon}^{1}\omega_{1}),Y_{0}+\sigma_{2}\xi(\theta_{\tau\epsilon}^{2}\omega_{2}))+g_{X}(X_{0}+\sigma_{1}\eta^{\epsilon}(\theta_{\tau\epsilon}^{1}\omega_{1}),Y_{0}+\sigma_{2}\xi(\theta_{\tau\epsilon}^{2}\omega_{2}))\\&(X(\tau\epsilon)-X_{0})+g_{Y}(X_{0}+\sigma_{1}\eta^{\epsilon}(\theta_{\tau\epsilon}^{1}\omega_{1}),Y_{0}+\sigma_{2}\xi(\theta_{\tau\epsilon}^{2}\omega_{2}))(Y(\tau\epsilon)-Y_{0}),\\=&g(X_{0}+\sigma_{1}\eta^{\epsilon}(\theta_{\tau\epsilon}^{1}\omega_{1}),Y_{0}+\sigma_{2}\xi(\theta_{\tau\epsilon}^{2}\omega_{2}))+g_{X}(X_{0}+\sigma_{1}\eta^{\epsilon}(\theta_{\tau\epsilon}^{1}\omega_{1}),Y_{0}+\sigma_{2}\xi(\theta_{\tau\epsilon}^{2}\omega_{2}))\\&\[\epsilon X_{1}(\tau)+\epsilon^{2}X_{2}(\tau)+\cdot\cdot\cdot\]+g_{Y}(X_{0}+\sigma_{1}\eta^{\epsilon}(\theta_{\tau\epsilon}^{1}\omega_{1}),Y_{0}+\sigma_{2}\xi(\theta_{\tau\epsilon}^{2}\omega_{2}))\[\epsilon\int_{0}^{\tau}[JY(s\epsilon)\\&+g(X(s\epsilon)+\sigma_{1}\eta^{\epsilon}(\theta_{s\epsilon}^{1}\omega_{1}),Y(s\epsilon)+\sigma_{2}\xi(\theta_{s\epsilon}^{2}\omega_{2}))]ds\].\end{align*}
Putting the Taylor expansion of $f$ and value of $X(\tau\epsilon)$ in (28),
\begin{align*}
&\frac{d\[X_{0}(\tau)+\epsilon X_{1}(\tau)+\epsilon^{2}X_{2}(\tau)+\cdot\cdot\cdot\]}{d\tau}\\=&A_{\alpha}\[X_{0}(\tau)+\epsilon X_{1}(\tau)+\epsilon^{2}X_{2}(\tau)+\cdot\cdot\cdot\]+f(X_{0}+\sigma_{1}\eta^{\epsilon}(\theta_{\tau\epsilon}^{1}\omega_{1}),Y_{0}+\sigma_{2}\xi(\theta_{\tau\epsilon}^{2}\omega_{2}))\\&+f_{X}(X_{0}+\sigma_{1}\eta^{\epsilon}(\theta_{\tau\epsilon}^{1}\omega_{1}),Y_{0}+\sigma_{2}\xi(\theta_{\tau\epsilon}^{2}\omega_{2}))\[\epsilon X_{1}(\tau)+\epsilon^{2}X_{2}(\tau)+\cdot\cdot\cdot\]+f_{Y}(X_{0}\\&+\sigma_{1}\eta^{\epsilon}(\theta_{\tau\epsilon}^{1}\omega_{1}),Y_{0}+\sigma_{2}\xi(\theta_{\tau\epsilon}^{2}\omega_{2}))\[\epsilon\int_{0}^{\tau}[JY(s\epsilon)+g(X(s\epsilon)+\sigma_{1}\eta^{\epsilon}(\theta_{s\epsilon}^{1}\omega_{1}),Y(s\epsilon)\\&+\sigma_{2}\xi(\theta_{s\epsilon}^{2}\omega_{2}))]ds\].
\end{align*}
Now, by comparing the terms with equal powers of $\epsilon$, it is concluded  that
\begin{align*}
\frac{dX_{0}(\tau)}{d\tau}=&A_{\alpha}X_{0}(\tau)+f(X_{0}+\sigma_{1}\eta^{\epsilon}(\theta_{\tau\epsilon}^{1}\omega_{1}),Y_{0}+\sigma_{2}\xi(\theta_{\tau\epsilon}^{2}\omega_{2})),\\&\mbox{ with initial value }X_{0}(0)=\mathcal{H}^{(0)}(\omega,Y_{0}),\\
\frac{dX_{1}(\tau)}{d\tau}=&\[A_{\alpha}+f_{X}(X_{0}+\sigma_{1}\eta^{\epsilon}(\theta_{\tau\epsilon}^{1}\omega_{1}),Y_{0}+\sigma_{2}\xi(\theta_{\tau\epsilon}^{2}\omega_{2}))\]X_{1}(\tau)+f_{Y}(X_{0}\\&+\sigma_{1}\eta^{\epsilon}(\theta_{\tau\epsilon}^{1}\omega_{1}),Y_{0}+\sigma_{2}\xi(\theta_{\tau\epsilon}^{2}\omega_{2}))\int_{0}^{\tau}[JY(s\epsilon)+g(X(s\epsilon)\\&+\sigma_{1}\eta^{\epsilon}(\theta_{s\epsilon}^{1}\omega_{1}),Y(s\epsilon)+\sigma_{2}\xi(\theta_{s\epsilon}^{2}\omega_{2}))]ds,
\\&\mbox{ with initial value } X_{1}(0)=\mathcal{H}^{(1)}(\omega,Y_{0}).\end{align*} We get the values of $X_{0}(\tau)$ and $X_{1}(\tau)$ by solving above two equations, i.e.,
\begin{align*}
X_{0}(\tau)=&e^{A_{\alpha}\tau}\mathcal{H}^{(0)}(\omega,Y_{0})+\int_{0}^{\tau}e^{A_{\alpha}(\tau-s)}f(X_{0}(s)+\sigma_{1}\eta^{\epsilon}(\theta_{s\epsilon}^{1}\omega_{1}),Y_{0}(s)+\sigma_{2}\xi(\theta_{s\epsilon}^{2}\omega_{2}))ds,\\
X_{1}(\tau)=&e^{A_{\alpha}\tau+\int_{0}^{\tau}f_{X}(X_{0}(s)+\sigma_{1}\eta^{\epsilon}(\theta_{s\epsilon}^{1}\omega_{1}),Y_{0}(s)+\sigma_{2}\xi(\theta_{s\epsilon}^{2}\omega_{2}))ds}\times\mathcal{H}^{(1)}(\omega,Y_{0})\\&+\int_{0}^{\tau}e^{A_{\alpha}(\tau-s)+\int_{s}^{\tau}f_{X}(X_{0}(r)+\sigma_{1}\eta^{\epsilon}(\theta_{r\epsilon}^{1}\omega_{1}),Y_{0}(r)+\sigma_{2}\xi(\theta_{r\epsilon}^{2}\omega_{2}))dr}
\\&\times f_{Y}(X_{0}(s)+\sigma_{1}\eta^{\epsilon}(\theta_{s\epsilon}^{1}\omega_{1}),Y_{0}(s)+\sigma_{2}\xi(\theta_{s\epsilon}^{2}\omega_{2}))\[\int_{0}^{s}[JY(r\epsilon)\\&+g(X(r\epsilon)+\sigma_{1}\eta^{\epsilon}(\theta_{r\epsilon}^{1}\omega_{1}),Y(r\epsilon)+\sigma_{2}\xi(\theta_{r\epsilon}^{2}\omega_{2}))]dr\]ds.
\end{align*}
From (18),
\begin{align*}\label{sde}\mathcal{H}^{\epsilon}(\omega, Y_{0})=&\frac{1}{\epsilon}\int_{-\infty}^{0}e^{-A_{\alpha}s/\epsilon}f(X(s)+\sigma_{1}\eta^{\epsilon}(\theta_{s}^{1}\omega_{1}),Y(s)+\sigma_{2}\xi(\theta_{s}^{2}\omega_{2}))ds,
\\=&\int_{-\infty}^{0}e^{-A_{\alpha}s}f(X(s\epsilon)+\sigma_{1}\eta^{\epsilon}(\theta_{s\epsilon}^{1}\omega_{1}),Y(s\epsilon)+\sigma_{2}\xi(\theta_{s\epsilon}^{2}\omega_{2}))ds,
\\=&\int_{-\infty}^{0}e^{-A_{\alpha}s}f(X_{0}+\sigma_{1}\eta^{\epsilon}(\theta_{s\epsilon}^{1}\omega_{1}),Y_{0}+\sigma_{2}\xi(\theta_{s\epsilon}^{2}\omega_{2}))ds\\&+\epsilon\int_{-\infty}^{0}e^{-A_{\alpha}s}\[f_{X}(X_{0}+\sigma_{1}\eta^{\epsilon}(\theta_{s\epsilon}^{1}\omega_{1}),Y_{0}+\sigma_{2}\xi(\theta_{s\epsilon}^{2}\omega_{2}))X_{1}(s)\\&+f_{Y}(X_{0}+\sigma_{1}\eta^{\epsilon}(\theta_{s\epsilon}^{1}\omega_{1}),Y_{0}+\sigma_{2}\xi(\theta_{s\epsilon}^{2}\omega_{2}))\times\int_{0}^{s}[JY(r\epsilon)+g(X(r\epsilon)\\&+\sigma_{1}\eta^{\epsilon}(\theta_{r\epsilon}^{1}\omega_{1}),Y(r\epsilon)+\sigma_{2}\xi(\theta_{r\epsilon}^{2}\omega_{2}))]dr\]ds+\mathcal{O}(\epsilon^{2}).
\end{align*}
Comparing above equation with equation (33), we find that
\begin{align*}
\mathcal{H}^{(0)}(\omega, Y_{0})=&\int_{-\infty}^{0}e^{-A_{\alpha}s}f(X_{0}+\sigma_{1}\eta^{\epsilon}(\theta_{s\epsilon}^{1}\omega_{1}),Y_{0}+\sigma_{2}\xi(\theta_{s\epsilon}^{2}\omega_{2}))ds,\\
\mathcal{H}^{(1)}(\omega,Y_{0})=&\int_{-\infty}^{0}e^{-A_{\alpha}s}\[f_{X}(X_{0}+\sigma_{1}\eta^{\epsilon}(\theta_{s\epsilon}^{1}\omega_{1}),Y_{0}+\sigma_{2}\xi(\theta_{s\epsilon}^{2}\omega_{2}))X_{1}(s)\\&+f_{Y}(X_{0}+\sigma_{1}\eta^{\epsilon}(\theta_{s\epsilon}^{1}\omega_{1}),Y_{0}+\sigma_{2}\xi(\theta_{s\epsilon}^{2}\omega_{2}))\times\int_{0}^{s}[JY(r\epsilon)\\&+g(X(r\epsilon)+\sigma_{1}\eta^{\epsilon}(\theta_{r\epsilon}^{1}\omega_{1}),Y(r\epsilon)+\sigma_{2}\xi(\theta_{r\epsilon}^{2}\omega_{2}))]dr\]ds.\end{align*}
So, the approximation of random slow manifold $\mathcal{M}^{\epsilon}(\omega)=\{(\mathcal{H}^{\epsilon}(\omega,Y_{0}),Y_{0})^{T}:Y_{0}\in H_{2}\}$ for random system
(9)-(10) up to order $\mathcal{O}(\epsilon^{2})$ is given by
\begin{align}
\mathcal{H}^{(\epsilon)}(\omega, Y_{0})=\mathcal{H}^{(0)}(\omega, Y_{0})+\epsilon\mathcal{H}^{(1)}(\omega, Y_{0})+\mathcal{O}(\epsilon^{2}).
\end{align}
Hence, the original system (1)-(2) has slow manifold $\mathcal{\tilde{M}}^{\epsilon}(\omega)=\{(\mathcal{\tilde{H}}^{\epsilon}(\omega,Y_{0}),Y_{0})^{T}:Y_{0}\in H_{2}\}$ up to order $\mathcal{O}(\epsilon^{2})$, where
\begin{align}
\mathcal{\tilde{H}}^{\epsilon}(\omega,Y_{0})=\mathcal{H}^{(0)}(\omega, Y_{0})+\epsilon\mathcal{H}^{(1)}(\omega, Y_{0})+\sigma_{1}\eta(\omega_{1})+\mathcal{O}(\epsilon^{2}).
\end{align}
\section{Examples}
\noindent \textbf{Example 1.}
Take a system
\begin{align}
&\dot{x}=\frac{1}{\epsilon}A_{\alpha}x+\frac{1}{6\epsilon}(y)^{2}+\frac{\sigma_{1}}{\sqrt[\alpha_{1}]{\epsilon}}\dot{L}_{t}^{\alpha_{1}}, \mbox{ in } H_{1}=L^{2}(-1,1),\\
&\dot{y}=y+\frac{1}{3}\sin\int_{-1}^{1}x(a)da+\sigma_{2}\dot{L}_{t}^{\alpha_{2}}, \mbox{ in } H_{2}=\mathbb{R},\end{align}
where $x$ is fast mode, $y$ is slow mode. While $\dot{L}_{t}^{\alpha_{1}}$ and $\dot{L}_{t}^{\alpha_{2}}$ are derivatives of scalar symmetric $\alpha$-stable L$\acute{e}$vy processes, with $1<\alpha<2$. Nonlinearities  $f=\frac{1}{6}(y)^{2}$ and $g=\frac{1}{3}\sin\int_{-1}^{1}x(a)da$ are Lipschitz continuous. Random system corresponding to stochastic system (36)-(37) is
\begin{align}
&\dot{X}=\frac{1}{\epsilon}A_{\alpha}X^{\epsilon}+\frac{1}{6\epsilon}((Y+\sigma_{2}\xi(\theta_{t}^{2}\omega_{2}))^{2}),\\
&\dot{Y}=Y+\frac{1}{3}\sin\left(\int_{-1}^{1}[X+\sigma_{1}\eta^{\epsilon}(\theta_{t}^{1}\omega_{1})]da\right).\end{align}
For sufficiently small $\epsilon>0$ and $Y_{0}\in \mathbb{R}$, random evolutionary system $(38)-(39)$ posses a random slow manifold, i.e., $$\mathcal{M}^{\epsilon}(\omega)=\{(\mathcal{H}^{\epsilon}(\omega,Y_{0}),Y_{0})^{T}:Y_{0}\in \mathbb{R}),$$
where
$$\mathcal{H}^{\epsilon}(\omega,Y_{0})=\frac{1}{6\epsilon}\int_{-\infty}^{0}e^{-A_{\alpha}s/\epsilon}\left(Y(s)+\sigma_{2}\xi(\theta_{s}^{2}\omega_{2})\right)^{2}ds.$$
Approximate slow manifold for nonlocal system (36)-(37) up to order $\mathcal{O}(\epsilon)$ is
$$\mathcal{\tilde{H}}^{\epsilon}(\omega,Y_{0})=H^{0}(\omega,Y_{0})+\sigma_{1}\eta^{\epsilon}(\omega_{1})+\mathcal{O}(\epsilon).$$
Where
$$H^{0}(\omega,Y_{0})=\frac{1}{6\epsilon}\int_{-\infty}^{0}e^{-A_{\alpha}s}(Y_{0}+\sigma_{2}\xi(\theta_{s\epsilon}^{2}\omega_{2}))^{2}ds.$$

\noindent \textbf{Example 2.}
Take a nonlocal fast-slow stochastic system
\begin{align}
&\dot{x}=\frac{1}{\epsilon}A_{\alpha}x+\frac{0.01}{\epsilon}(\sqrt{y^{2}+5}-\sqrt{5})+\frac{\sigma_{1}}{\sqrt[\alpha_{1}]{\epsilon}}\dot{L}_{t}^{\alpha_{1}}, \mbox{ in } H_{1}=L^{2}(-1,1),\\
&\dot{y}=-y+(0.01\times b)\sin\int_{-1}^{1}xda+\sigma_{2}\dot{L}_{t}^{\alpha_{2}}, \mbox{ in } H_{2}=\mathbb{R},\end{align}
where $x$ is fast mode, $y$ is slow mode, $a$ and $b$ are positive real unknown parameter. While $\dot{L}_{t}^{\alpha_{1}}$ and $\dot{L}_{t}^{\alpha_{2}}$ are derivatives of scalar symmetric $\alpha$-stable L$\acute{e}$vy processes, with $1<\alpha<2$. Lipschitz continuous nonlinearities are $f=0.01(\sqrt{y^{2}+5}-\sqrt{5})$ and $g=(0.01\times b)\int_{-1}^{1}xda$. Lipschitz constants of $f$ and $g$ are $L_{f}=0.01$ and $L_{g}=0.01\times b$ respectively. Random system corresponding to stochastic system (40)-(41):
\begin{align}
&\dot{X}=\frac{1}{\epsilon}A_{\alpha}X+\frac{0.01}{\epsilon}(\sqrt{(Y+\sigma_{2}\xi(\theta_{t}^{2}\omega_{2}))^{2}+5}-\sqrt{5}),\\
&\dot{Y}=-Y+(0.01\times b)\sin\left(\int_{-1}^{1}[X+\sigma_{1}\eta^{\epsilon}(\theta_{t}^{1}\omega_{1})]da\right).\end{align}
For sufficiently small $\epsilon>0$, random system $(42)-(43)$ posses a exponential tracking slow manifold,  $$\mathcal{M}^{\epsilon}(\omega)=\{(\mathcal{H}^{\epsilon}(\omega,Y_{0}),Y_{0})^{T}:Y_{0}\in \mathbb{R}),$$
where
$$\mathcal{H}^{\epsilon}(\omega,Y_{0})=\frac{0.01}{\epsilon}\int_{-\infty}^{0}e^{-A_{\alpha}s/\epsilon}(\sqrt{(Y_{0}+\sigma_{2}\xi(\theta_{t}^{2}\omega_{2}))^{2}+5}-\sqrt{5})ds.$$
Approximate slow manifold for nonlocal system (41)-(42) up to order $\mathcal{O}(\epsilon)$ is
$$\mathcal{\tilde{H}}^{\epsilon}(\omega,Y_{0})=H^{0}(\omega,Y_{0})+\sigma_{1}\eta^{\epsilon}(\omega_{1})+\mathcal{O}(\epsilon).$$
Where for a fixed $Y_{0}\in \mathbb{R}$,
$$H^{0}(\omega,Y_{0})=\frac{0.01}{\epsilon}\int_{-\infty}^{0}e^{-A_{\alpha}s}(\sqrt{(Y_{0}+\sigma_{2}\xi(\theta_{s\epsilon}^{2}\omega_{2}))^{2}+5}-\sqrt{5})ds.$$
We have conducted the numerical simulation for example 2. The simulation of example 1 is similar, so we omit that.
\newpage
\begin{figure}[http]
  \centering
  \begin{minipage}[b]{0.4\textwidth}
  \includegraphics[width=6cm,height=4cm]{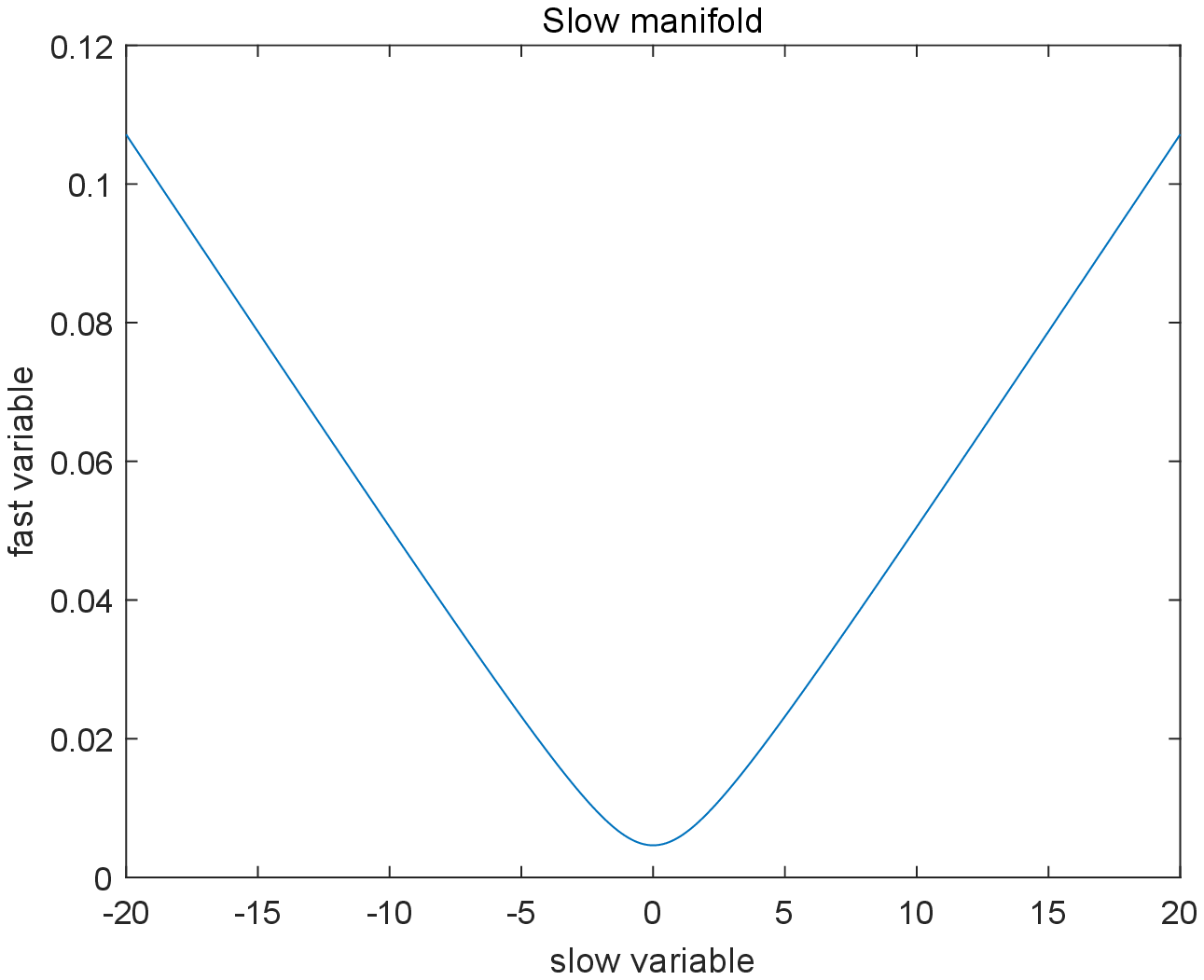}
   \label{fig:graph1}
 \end{minipage}
 \hfill
 \begin{minipage}[b]{0.4\textwidth}
  \includegraphics[width=6cm,height=4cm]{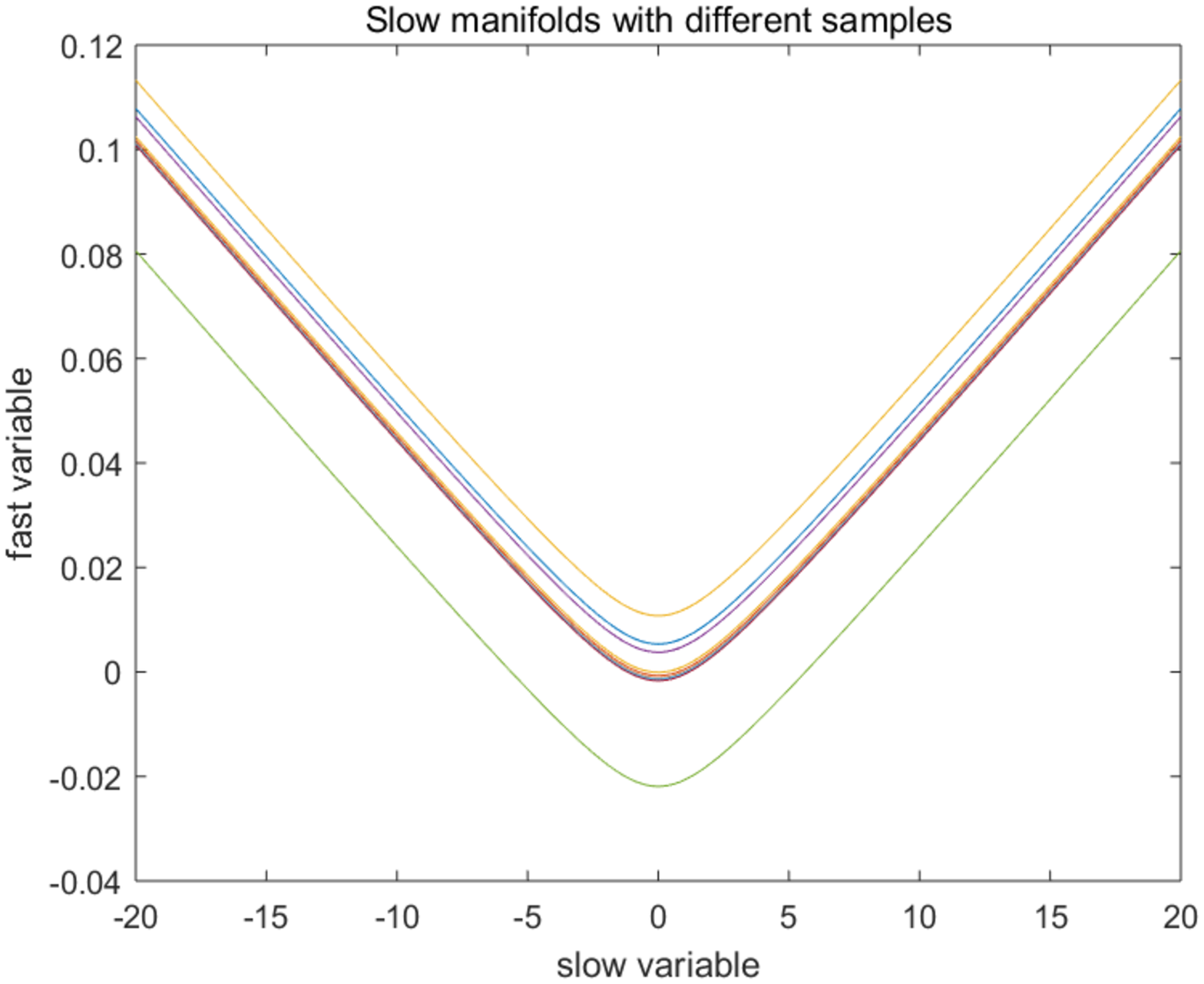}
  \label{fig:graph1}

  \end{minipage}
 \caption{ (left) Random slow manifold for one sample, (right) random slow manifold for different samples.}
\end{figure}

\begin{figure}[http]
  \centering
  \begin{minipage}[b]{0.4\textwidth}
  \includegraphics[width=6cm,height=4cm]{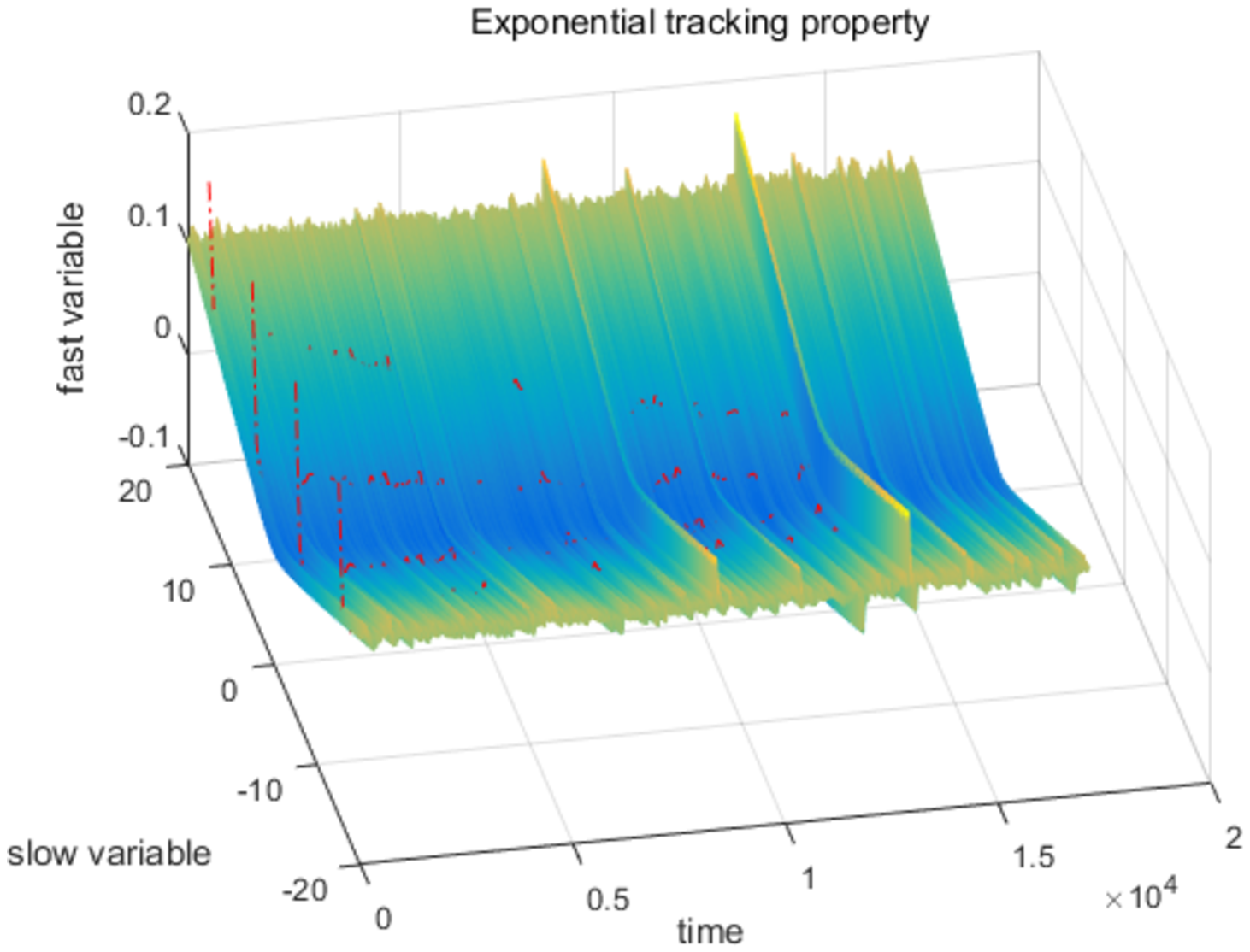}
   \label{fig:graph1}
 \end{minipage}
 \hfill
 \begin{minipage}[b]{0.4\textwidth}
  \includegraphics[width=6cm,height=4cm]{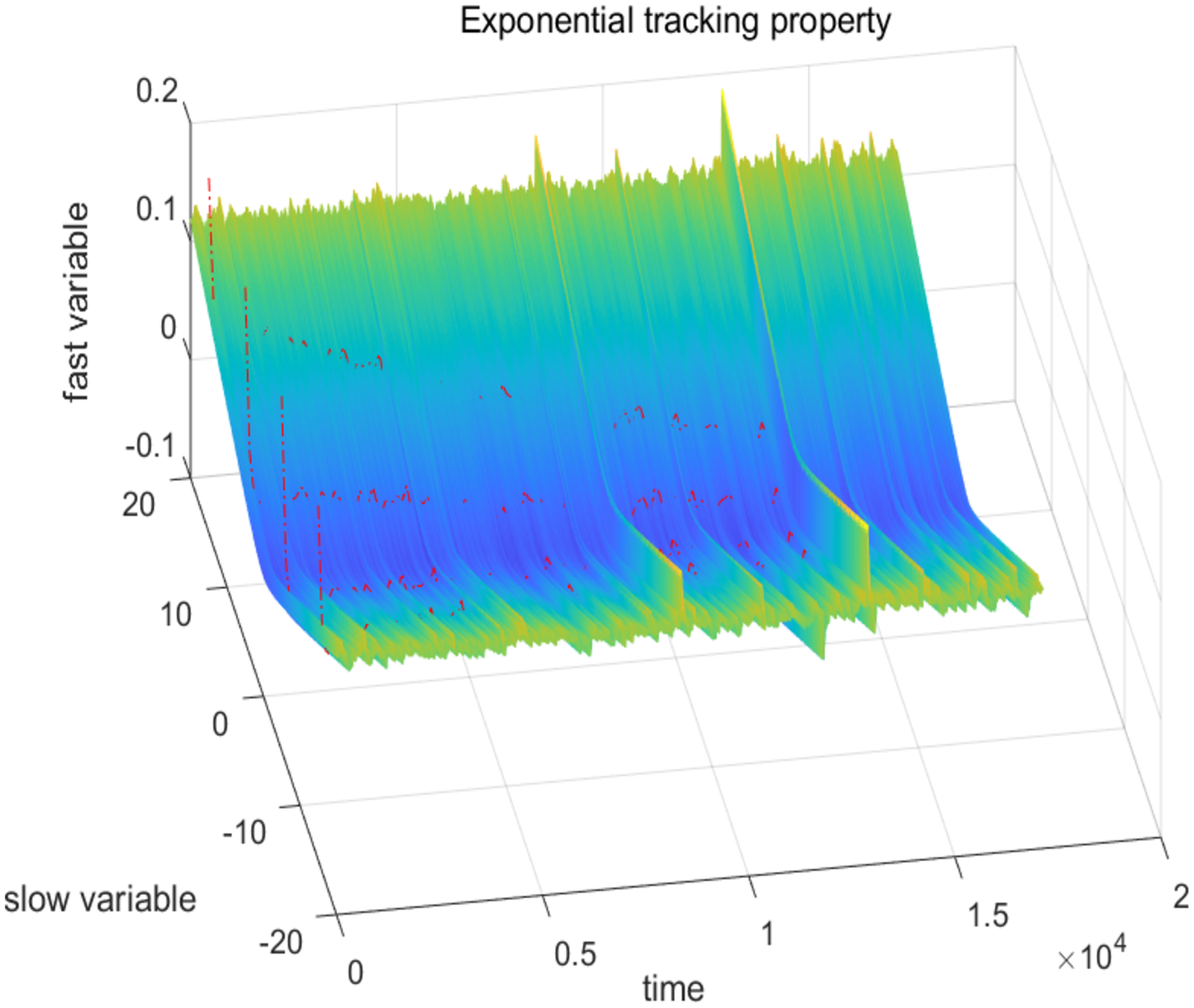}
  \label{fig:graph1}

  \end{minipage}
 \caption{ (left) Exponential tracking property in the system for $\alpha=1.2$ and $\epsilon=0.01$, (right) exponential tracking property in the system for $\alpha=1$ and $\epsilon=0.01$.}
\end{figure}
%\begin{figure}[http]
 % \centering
 % \includegraphics[width=6cm,height=4cm]{etp.eps}
  %\caption{Inter system exponential tracking property when $\sigma_{1}=\sigma_{2}=0.1,\alpha=1.2$ and $\epsilon=0.01$.}
%\end{figure}
%\newpage
%\begin{figure}[http]
  %\centering
  %\begin{minipage}[b]{0.4\textwidth}
  %\includegraphics[width=6cm,height=4cm]{PE1.eps}
  % \label{fig:graph1}
 %\end{minipage}
 %\hfill
 %\begin{minipage}[b]{0.4\textwidth}
  %\includegraphics[width=6cm,height=4cm]{PE2.eps}
  %\label{fig:graph1}

  %\end{minipage}
 %\caption{  (left) Parameter estimation based on original fast-slow nonlocal system and (right) parameter estimation based slow system: $\alpha=1.5$ and $\epsilon=0.03$.}
%\end{figure}
%\begin{figure}[http]
 % \centering
  %\begin{minipage}[b]{0.4\textwidth}
  %\includegraphics[width=6cm,height=4cm]{PE3.eps}
   %\label{fig:graph1}
 %\end{minipage}
 %\hfill
 %\begin{minipage}[b]{0.4\textwidth}
 % \includegraphics[width=6cm,height=4cm]{PE4.eps}
  %\label{fig:graph1}

  %\end{minipage}
 %\caption{ (left) Parameter estimation based on original fast-slow nonlocal system and (right) parameter estimation based on slow system: $\alpha=1.2$ and $\epsilon=0.01$.}
%\end{figure}

\bibliographystyle{line}
\bibliography{JAMS-paper}

\end{document}